\documentclass{amsart}

\usepackage{amssymb}
\usepackage{amsmath}
\usepackage{verbatim}
\usepackage{url}
\usepackage{graphicx}
\usepackage{color}
\usepackage[shortlabels]{enumitem}

\newcommand{\eps}{\varepsilon}
\newcommand{\opnm}{\operatorname}
\newcommand{\BigO}{\mathcal{O}}

\newcommand{\C}{\mathbb{C}}
\newcommand{\R}{\mathbb{R}}
\newcommand{\N}{\mathbb{N}}
\newcommand{\Dom}{\operatorname{Dom}}
\newcommand{\Num}{\operatorname{Num}}

\newtheorem{theorem}{Theorem}[section]
\newtheorem{proposition}[theorem]{Proposition}
\newtheorem{corollary}[theorem]{Corollary}
\newtheorem{lemma}[theorem]{Lemma}
\theoremstyle{definition}

\theoremstyle{remark}
\newtheorem{remark}[theorem]{Remark}

\newtheorem{asm}{Assumption}

\newtheorem*{acknowledgements}{Acknowledgements}

\numberwithin{equation}{section}

\title[Various rates of projection growth]{Differential operators admitting various rates of spectral projection growth}

\author{Boris Mityagin}
\address[Boris Mityagin]{Department of Mathematics, The Ohio State University, 231 West 18th Ave.,
Columbus, OH 43210, USA}
\email{mityagin.1@osu.edu}

\author{Petr Siegl}
\address[Petr Siegl]{Mathematical Institute, University of Bern, Alpeneggstrasse 22, 3012 Bern, Switzerland; On leave from Nuclear Physics Institute ASCR, 25068 \v Re\v z, Czech Republic}
\email{petr.siegl@math.unibe.ch}

\author{Joe Viola}
\address[Joe Viola]{Laboratoire de Math\'ematiques J. Leray, UMR 6629 du CNRS, Universit\'e de Nantes, 2, rue de la Houssini\`ere, 44322 Nantes Cedex 03, France}
\email{Joseph.Viola@univ-nantes.fr}

\begin{document}

\begin{abstract}
We consider families of non-self-adjoint perturbations of the self-adjoint Schr\"odinger operators with single-well potentials. The norms of spectral projections of these operators are found to grow at intermediate rates from arbitrarily slowly to exponentially rapidly.
\end{abstract}

\keywords{harmonic and anharmonic oscillators, Hermite functions, spectral projections, Riesz basis. \\ 
\copyright 2016. This manuscript version is made available under the CC-BY-NC-ND 4.0 license \url{http://creativecommons.org/licenses/by-nc-nd/4.0/}}

\maketitle

\section{Introduction}

We consider operators $T$ on separable Hilbert spaces $H$; in particular, the operators studied are unbounded, densely defined, and have compact resolvents. In this case, the spectrum of $T$ consists of at most a countable number of isolated eigenvalues accumulating at infinity. For an eigenvalue $\lambda$, we may take an $\eps > 0$ smaller than the distance from $\lambda$ to any other eigenvalue, and we define the spectral projection $P_{\lambda}$ via the Cauchy integral formula
\begin{equation}\label{e1ProjDef}
	P_{\lambda} = \frac{1}{2\pi i}\int_{|\zeta - \lambda| = \eps} (\zeta-T)^{-1}\,d\zeta.
\end{equation}
See, for instance,\ \cite[Thm.\ VII.3.18, VII.4.5]{DunSch1988}.
The range of $P_\lambda$ is a finite-dimensional space consisting of root vectors of $T$ corresponding to $\lambda$. 

We follow \cite[Sec.\ 3.3]{Davies-2007} in distinguishing between collections of vectors $\{f_j\}_{j=0}^\infty$ in a Hilbert space $\mathcal{H}$ which are complete, minimal complete, or bases. A set is complete if the closure of its linear span is all of $\mathcal{H}$; a minimal complete set is a complete set which is no longer complete if any one of the vectors is removed; and a basis is one where each $f \in \mathcal{H}$ admits a unique sequence of scalars $\{\alpha_j\}_{j=0}^\infty$ with
\[
	f = \lim_{n\rightarrow\infty} \sum_{j=0}^n \alpha_j f_j. 
\]

In constrast with the case where $T$ is normal, the spectral projections of a non-normal operator $T$ are not, in general, orthogonal. We study simple differential operators on the real line which admit minimal complete systems of eigenvectors $\{u_k\}_{k=0}^\infty$ which are not bases.  If an operator has a family of spectral projections with norms going to infinity, its eigenvectors cannot form a basis; see for instance \cite[Lem.\ 3.3.3]{Davies-2007}.  For more detailed information, we consider the question of asymptotics for the norms of spectral projections. Our operators act on the Hilbert space of functions $L^2(\R)$, with norm
\[
 	\|f\|_{L^2} = \left(\int_{-\infty}^\infty |f(x)|^2 \,dx\right)^{1/2}.
\]

Rapid growth of norms of spectral projections for non-self-adjoint differential operators is well-known.  In \cite[Thm.\ 6, 10]{Davies-2000}, Davies shows growth more rapid than any power of $k$ for the norms of spectral projections $\{P_k\}_{k=0}^\infty$ for the operators
\[
 	T = -\frac{d^2}{dx^2} + c|x|^m,
\]
acting on $L^2(\R)$, with $m > 0$, $c \in \C\backslash \R$, and $|\arg c| < C(m)$ some positive $m$-dependent constant. The exact exponential rate of spectral projection growth for the complex harmonic oscillators
\[
 	T = -\frac{d^2}{dx^2} + z^4x^2 = -\frac{d^2}{dx^2} + x^2 + (z^4-1)x^2, \quad |\opnm{arg} z| < \pi/4,
\]
also acting on $L^2(\R)$, was computed by Davies and Kuijlaars in \cite{DavKui2004}.  Exponential rates of growth for a wider variety of operators in one dimension were computed by Henry in \cite{Hen2012} and \cite{Hen2013},
including
\[
 	-\frac{d^2}{dx^2} + e^{i\theta}x^m,
\]
where either $m = 1$ or $m$ is an even integer with $m \geq 2$, and where $\theta$ obeys
 $$
 |\theta| < \min 
 \left\{
 \frac{(m+2)\pi}{4}, \frac{(m+2)\pi}{2m}
 \right\}.
 $$

In these cases, the eigenvalues of the operators are simple and have modulus tending towards $\infty$; write $\{\lambda_k\}_{k=0}^\infty$ for the eigenvalues ordered by increasing modulus and $P_k$ for the corresponding spectral projections.  The results in \cite{DavKui2004} and \cite{Hen2012} give exponential growth in $k$ of the $L^2(\R)$ operator norms of the projections $P_k$, which is of the form
\begin{equation}\label{max.gr}
 	\lim_{k\rightarrow \infty} \frac{1}{k}\log\|P_k\| = c
\end{equation}
for some $c \in (0, \infty)$.

Our focus here is to demonstrate the existence of natural classes of operators for which
\begin{equation}\label{e1GrowthsVarious}
 	\lim_{k\rightarrow \infty} \frac{1}{k^\sigma}\log\|P_k\| = c
\end{equation}
with $c \in (0, \infty)$ and where $\sigma$ can take any value in $(0,1)$. This can be found in Theorem \ref{t3.2}. Slower rates of growth, such as polynomial, may be obtained as well; see Theorems \ref{t3} and \ref{t3.2} and Corollaries \ref{cor:Pk.beta} and \ref{cor:Pk.beta.slow}. On the other hand, we have no examples demonstrating more rapid growth than \eqref{max.gr}.

These operators arise as non-symmetric perturbations $B$ of self-adjoint Schr{\"o}dinger operators $T$. The perturbations are relatively bounded with respect to the self-adjoint operator, meaning that
\begin{equation}\label{e1relb}
\begin{aligned}
& \Dom(T) \subset \Dom(B),\\
&	\forall f \in \Dom(T), \quad \|Bf\| \leq a \|f\| + b \|Tf\|, \quad  a,b \geq 0.
\end{aligned}
\end{equation}
The infimum of all possible $b$ is called the $T$-bound; see \cite[Sec.\ IV.1.1]{Kato-1966}.  The perturbations considered in Sections \ref{sec.ho} obey the stronger condition that they are $s$-subordinated to $T$, which is to say that
\begin{equation}\label{e1sub}
\begin{aligned}
& \Dom(T) \subset \Dom(B),\\
&	\forall f \in \Dom(T), \quad \|Bf\| \leq C\|Tf\|^s \|f\|^{1-s}, \quad 0 \leq s <1, \quad C > 0;
\end{aligned}
\end{equation}
see \cite[Sec.\ 1.5, 1.9]{Markus-1988}. Note that this property implies that $B$ is $T$-bounded with $T$-bound zero.

The relative boundedness of $B$ with respect to $T$ is important in several steps of our analysis, including showing that $L = T+B$, like the unperturbed operator $T$, has a compact resolvent and a complete system of root vectors, which we explicitly describe. In the applications to anharmonic oscillators, such as those given by the potentials $q(x)=|x|^\beta$, $\beta>2$, and specially chosen $B$, we even have $s$-subordination of $B$; see Remark \ref{rem:s.sub}.
For these operators, we have growth of spectral projection norms which is slower than the growths found in \cite{DavKui2004} and \cite{Hen2012}, where the skew-adjoint parts of the operators are not $s$-subordinated to the self-adjoint parts for any $0 \leq s < 1$. Finally, if the perturbations satisfy certain local subordination conditions, then the norms of spectral projections are bounded and the eigensystem of $L$ contains a Riesz basis; see
Section \ref{subsec:RB} for more details.

The plan of this paper is follows.  In Section \ref{sec.ho}, we consider first-order perturbations of the harmonic oscillator which arise naturally and may be understood via elementary applications of the Fourier transform and various facts about the Hermite functions. This section is less technical and could be read independently. Afterwards, in Section \ref{s3Monomial}, we introduce a family of perturbations of Schr\"odinger operators with single-well potentials growing more rapidly than $x^2$ as $x \to \infty$. These perturbations are crafted to give a new family of eigenfunctions which are multiples of the eigenfunctions of the Schr\"odinger operators, and asymptotics for these eigenfunctions give quite varied rates of growth.  Finally, in Section \ref{s4Remarks}, we indicate further generalizations, examples, and related results.

\begin{acknowledgements}
P.\ S.\ acknowledges the support of SCIEX-NMS Fellowship 11.263 and SNF Ambizione project PZ00P2\_154786.  
J.\ V.\ is grateful for the support of ANR NOSEVOL (Project number: ANR 2011 BS01019 01). The authors acknowledge the very pleasant working conditions at the ICMS Workshop \emph{Mathematical aspects of the physics with
non-self-adjoint operators}, held in March 2013 in Edinburgh, UK, during which this work was initiated.  The authors would also like to thank Andrei Mart\`inez-Finkelshtein for very valuable advice on references and calculations in Proposition \ref{p4Even}, Section \ref{s4Remarks}.  The authors would like to thank Rapha\"el Henry for bringing an important concern to our attention. Finally, the authors would like to express their gratitude to the anonymous referee for a careful reading and many helpful suggestions.

\end{acknowledgements}

\section{Perturbations of the harmonic oscillator}
\label{sec.ho}

\subsection{Definition of the operators}

Our first object of study is the harmonic oscillator in dimension one,
\begin{equation}\label{e2TDef}
\begin{aligned}
		T = -\frac{d^2}{dx^2} + x^2, \quad 
  \Dom(T) = \{ f \in W^{2,2}(\R) \::\: x^2f \in L^2(\R) \}
\end{aligned}
\end{equation}
and the non-self-adjoint operator
\begin{equation}
	L = T+B
\end{equation}
with the skew-adjoint perturbation
\begin{equation}\label{e2BDef}
\begin{aligned}
 	B &= 2iax, \quad a \in \R\backslash \{0\},\\
 \Dom(B) & =\{f \in L^2(\R): xf \in L^2(\R)\}.	
\end{aligned} 	
\end{equation}
See \cite[Sec.\ 3.7]{Davies-1995} for a discussion of the Sobolev spaces $W^{n,2}(\R)$.
Despite the $s$-subordination of $B$ to $T$, significant effects from the nonnormality of $L$ appear. Specifically, we show that the spectral projections of $L$ satisfy \eqref{e1GrowthsVarious} with $\sigma = 1/2$.

\subsection{Elementary facts}

The harmonic oscillator \eqref{e2TDef} is a well-studied operator which plays a fundamental role in physics as the model for a Schr{\"o}dinger operator near a minimum of the potential.  Its spectrum is the set of all positive odd integers,
\[
 	\opnm{Spec} (T) = \{2k +1 \::\: k=0, 1, 2,\dots \},
\]
with each point in the spectrum being an eigenvalue of multiplicity one. The eigenfunctions are also very well understood. The Hermite polynomials,
\begin{equation}\label{e2HkDef}
 	H_k(x) = e^{x^2/2}\Big(x-\frac{d}{dx}\Big)^k e^{-x^2/2}, \quad k=0, 1, 2, \dots,
\end{equation}
are polynomials of degree $k$.  Furthermore, each $H_k$ is either an odd or even function, depending on whether $k$ is odd or even.  An orthonormal basis for $L^2(\R)$ consisting of eigenfunctions of $T$ is given by the Hermite functions,
\begin{equation}\label{e2hkDef}
 	h_k(x) = \frac{1}{\sqrt{2^k k! \sqrt{\pi}}}H_k(x) e^{-x^2/2}, \quad k = 0, 1, 2,\dots,
\end{equation}
which obey
\[
 	Th_k = (2k+1)h_k, \quad k = 0, 1, 2, \dots.
\]
As $|x|\rightarrow \infty$, the Gaussian factor $e^{-x^2/2}$ gives very rapid decay to each fixed Hermite function, uniformly in any strip
\[
 	\{z \in \C \::\: |\Im z| \leq A\}, \quad A>0. 
\]
In more detail, if $|a| \leq A$, then 
\begin{equation}\label{hk.strip}
|h_k(x+i a)| \leq  e^{-(x^2 - A^2)/2} |H_k(x+ia)| 
\leq C(\varepsilon, k) \, e^{-((1-\eps)x^2 - A^2)/2},
\end{equation}
where $\varepsilon > 0$ can be taken arbitrarily small and $C(\varepsilon, k)>0$.

We also consider the action of the Fourier transform,
\[
	\mathcal{F}f(\xi) = \hat{f}(\xi) = \frac{1}{\sqrt{2\pi}} \int_{-\infty}^\infty e^{-ix\xi} f(x)\,dx.
\]
The Gaussian $G(x) = e^{-x^2/2}$ is invariant under the Fourier transform, $\hat{G}(\xi) = G(\xi)$, and conjugation with the Fourier transform interchanges multiplication by $x$ with $i \frac{d}{d \xi}$ and $-i \frac{d}{d x}$ with multiplication by $\xi$. Taking all this with \eqref{e2HkDef} and \eqref{e2hkDef} gives 
\begin{equation}\label{e2hkTildeSame}
	\hat{h}_k(\xi) = (-i)^k h_k(\xi).
\end{equation}
Moreover, the operator $L$ is unitarily equivalent to 
\begin{equation}\label{e2tildeBDef}
 \tilde L = T+\tilde B, \quad 	\tilde{B} = 2a \frac{d}{dx} .
\end{equation}

\subsection{Definition of the perturbed operator}
\label{subsec.def.ho.L}

We show that the perturbation $B$ is $s$-subordinated to $T$ as in \eqref{e1sub} in Lemma \ref{lem.ho.sub} below. (A more general statement is given in \cite[Prop.\ 7, Cor.\ 8]{AddMitDja2010}.) To begin, let us prove a fact about the graph norm of $T$. This equivalence is quite standard, and may be found in, for instance, \cite[Ex.\ 7.2.4]{Blank-Exner-Havlicek-2008}.

\begin{lemma}\label{lem.ho.grn}
There exist positive constants $k_1,k_2$ such that, for all $f\in \Dom(T)$, 
\begin{equation}\label{ho.grn.eq}
\begin{aligned}
k_1 \left(
\|f''\|^2 + \|x^2 f\|^2 + \|f\|^2 
\right) 
&\leq 
\|Tf\|^2 + \|f\|^2 \\
&\leq
k_2 \left(
\|f''\|^2 + \|x^2 f\|^2 + \|f\|^2 
\right).
\end{aligned}
\end{equation}
\end{lemma}

\begin{remark}\label{rem.ho.dom}
 	As a consequence of Lemma \ref{lem.ho.grn}, we see that we can equivalently write
	\[
	 	\Dom(T) = \{f \in L^2(\R)\::\: Tf \in L^2(\R)\}.
	\]
\end{remark}

\begin{proof}
We directly compute that
\[
 	\|Tf\|^2 = \|f''\|^2 + \|x^2 f\|^2 - 2\Re \langle f'', x^2f\rangle.
\]
By the Cauchy-Schwarz inequality, we immediately have the right-hand side inequality in \eqref{ho.grn.eq} with $k_2 = 2$.  Using integration by parts,
\[
 	\|Tf\|^2 = \|f''\|^2 + \|x^2 f\|^2 + 2\|xf'\|^2 + 4 \Re \langle f', xf\rangle \geq \|f''\|^2 + \|x^2 f\|^2 - 2\|f'\|^2 - 2\|xf\|^2.
\]
Completing the square shows that
\[
 	\|x^2f\|^2 - 4\|xf\|^2 = \|(x^2-2)f\|^2 - 4\|f\|^2 \geq -4\|f\|^2,
\]
and using the Fourier transform gives the corresponding fact that, when $f \in W^{2,2}(\R)$,
\[
 	\|f''\|^2 - 4\|f'\|^2  = \|\xi^2\hat{f}\|^2 - 4\|\xi \hat{f}\|^2 \geq -4\|f\|^2.
\]
Therefore, for all $f \in \Dom(T)$,
\begin{align*}
 	\|Tf\|^2 &+ \|f\|^2 \geq \frac{2}{9}\|Tf\|^2 + \|f\|^2 
	\\ & \geq \frac{1}{9}(\|f''\|^2 + \|x^2f\|^2) + \frac{1}{9}(\|x^2f\|^2 - 4\|xf\|^2 + \|f''\|^2 - 4\|f'\|^2) + \|f\|^2
	\\ & \geq \frac{1}{9}(\|f''\|^2 + \|x^2f\|^2+\|f\|^2),
\end{align*}
proving the left-hand side inequality in \eqref{ho.grn.eq} with $k_1 = 1/9$.
\end{proof}

\begin{lemma}\label{lem.ho.sub}
Let $T$ and $B$ be the operators defined in \eqref{e2TDef} and \eqref{e2BDef}. Then $B$ is $s$-subordinated to $T+1$ for $s = 1/2$; that is, there exists $C > 0$ such that
\begin{equation}\label{ho.oper.sub}
\forall f \in \Dom(T), \quad 
\|Bf\| 
\leq
C \|(T+1)f\|^{1/2} \|f\|^{1/2}.
\end{equation}
\end{lemma}
\begin{proof}
We note that, since $T$ is clearly a positive-definite operator, we have
\begin{equation*}
\|(T+1)f\|^2 = \|Tf\|^2 + \|f\|^2 + 2 \langle Tf, f\rangle 
\geq \|Tf\|^2 + \|f\|^2.
\end{equation*}
Therefore, using the equivalence of norms \eqref{ho.grn.eq},
\begin{align*}
\|xf\| &\leq \|x^2f\|^{1/2}\|f\|^{1/2} 
\leq 
\left(
\|f''\|^2 + \|x^2f\|^2 + \|f\|^2
\right)^{1/4}\|f\|^{1/2} \\
&\leq
C \|(T+1)f\|^{1/2} \|f\|^{1/2}. 
\qedhere
\end{align*}
\end{proof}

It follows immediately from Lemma \ref{lem.ho.sub} that $B$ is also relatively bounded with respect to $T$ with $T$-bound equal to zero, that is, for every $\varepsilon >0$ there exists a positive constant $C(\varepsilon)$ such that
\begin{equation*}
\forall f \in \Dom(T), \quad 
\|Bf\| 
\leq
\varepsilon \|Tf\| + C(\varepsilon) \|f\|.
\end{equation*}
Hence, the operator sum 
\begin{equation}\label{e2LDef}
\begin{aligned}
L & = T + B, \quad 
\Dom(L) = \Dom(T).
\end{aligned}
\end{equation}
is a closed operator; see \cite[Thm.\ IV.1.1]{Kato-1966}. 
It also follows easily that $L$ has a compact resolvent, by recalling that the resolvent of $T$ is compact and applying \cite[Thm.\ IV.1.16]{Kato-1966} considering $L+r$ as a perturbation of $T+r$ for $r > 0$ sufficiently large.

It is also easy to see that the adjoint of $L$ may be written
\begin{equation}\label{L.ho.act}
\begin{aligned}
L^* &= -\frac{d^2}{dx^2} + x^2  - 2 i a x, \quad 
\Dom(L^*) & = \Dom(T).
\end{aligned}
\end{equation}

\subsection{Characterization of the eigenfunctions}\label{ss2char}

Define
\begin{equation}\label{fk.gk.def}
	f_k(x) = h_k(x+ia), \qquad
	g_k(x) = h_k(x-ia),
\end{equation}
using the fact that the Hermite functions $h_k(x)$ are entire functions.

\begin{lemma}\label{l2Eigenfunctions}
	The functions $f_k$, $g_k$ defined in \eqref{fk.gk.def} belong to 
	\[
		\Dom(L) = \Dom(L^*) =\Dom(T),
	\]
	defined in \eqref{e2TDef}, and satisfy the eigenfunction relations
	\begin{equation}\label{e2fkEigs}
	\begin{aligned}
	Lf_k &= (2k +1 + a^2) f_k, \\
	L^*g_k &= (2k + 1 + a^2) g_k.
	\end{aligned}
	\end{equation}
	Furthermore, 
	\begin{equation}\label{e2fkgkNorms}
		\|f_k\| = \|g_k\| = \|e^{ax}h_k\|.
	\end{equation}
\end{lemma}

\begin{proof}
	The first claim follows from simple calculations and estimates like \eqref{hk.strip}.
	We note that
	\[
		L = -\frac{d^2}{dx^2} + x^2 + 2iax = -\frac{d^2}{dx^2} + (x+ia)^2 + a^2,
	\]
	and similarly for $L^*$, hence \eqref{e2fkEigs} follows.  
	
	The norm characterization \eqref{e2fkgkNorms} follows from taking the Fourier transform.  We are allowed to make a complex change of variables given the decay estimate \eqref{hk.strip} and that $h_k$ is holomorphic on all of $\C$.  Therefore
	\[
		\hat{f}_k(\xi) = \frac{1}{\sqrt{2\pi}}\int_{-\infty}^\infty e^{-ix\xi}h_k(x+ia)\,dx = \frac{1}{\sqrt{2\pi}}\int_{-\infty}^\infty e^{-i(x-ia)\xi}h_k(x)\,dx = e^{-a\xi}\hat{h}_k(\xi).
	\]
	Each $h_k$ is either odd or even, and by \eqref{e2hkTildeSame}, we obtain \eqref{e2fkgkNorms} for $f_k$.  The same statement for $g_k$ follows from replacing $a$ with $-a$.
\end{proof}

\begin{lemma}\label{l2Complete}
	The collections $\{f_k\}$ and $\{g_k\}$ defined in \eqref{fk.gk.def} are complete.  For $L$ defined in \eqref{e2LDef},
	\[
		\opnm{Spec} (L) = \opnm{Spec} (L^*) = \{2k + 1 + a^2 \::\: k = 0, 1, 2, \dots\},
	\]
	each point being a simple eigenvalue, and thus every root vector of $L$ is a multiple of an $f_k$ and every root vector of $L^*$ is a multiple of a $g_k$.
\end{lemma}

\begin{proof}
	Interchanging $a$ with $-a$ interchanges $L$ and $L^*$, so we need only consider the claim about $L$ and the system $\{f_k\}$. Let us prove that $(\{f_k\}_{k=0}^\infty)^\perp = \{0\}$.  Consider $v \in (\{f_k\}_{k = 0}^\infty)^\perp$, meaning that, for each $k$,
	\[
	0=\langle f_k, v\rangle = \langle \hat{f}_k, \hat{v}\rangle = (-i)^k \int_{-\infty}^\infty H_k(\xi) e^{-\xi^2/2 - a\xi}\overline{\hat{v}(\xi)}\,d\xi.
	\]
	The linear span of the Hermite polynomials $\{H_k\}$ defined in \eqref{e2HkDef} forms the set of all polynomials. Therefore, for every $k \geq 0$,
\begin{equation}\label{e2ZeroMomenta}
 	\int_{-\infty}^\infty \overline{\hat{v}(\xi)} \xi^k e^{-a\xi-\xi^2/2}\,d\xi = 0.
\end{equation}

	Since the function 
	\[
		\hat{w}(\xi) = \overline{\hat{v}(\xi)}e^{-a\xi-\xi^2/2}
	\]
	decays rapidly near infinity in any strip $\{|\Im\xi| \leq A\}$, its inverse Fourier transform
	\[
		\mathcal{F}^{-1}\hat{w}(x) = w(x) = \frac{1}{\sqrt{2\pi}}\int_{-\infty}^\infty e^{ix\xi}\hat{w}(\xi)\,d\xi
	\]
	may be extended to a holomorphic function on all of $\C$.  But \eqref{e2ZeroMomenta} shows that
	\[
		\frac{d^k w}{dx^k} (0) = 0, \quad k = 0, 1, 2, \dots.
	\]
	Since $w(x)$ is everywhere equal to its Taylor series, $w = 0$ on all of $\R$.  We conclude that $v = 0$, as claimed.

	We next show that the completeness of $\{f_k\}_{k=0}^\infty$ implies that there can be no other root vectors of $L$.  Let $P_{k}$ denote the spectral projection associated with the eigenvalue $2k+1+a^2$, so
	\[
	 	P_{k} f_k = f_k.
	\]
	Assume that there exists some $\mu \in \opnm{Spec}(L) \backslash \{2k + 1 +a^2 \, : \, k= 0,1,2, \dots \}$, and write $Q$ for the associated spectral projection.  Then $QP_k = 0$ and $Qf_k = QP_kf_k = 0$ for each $k$, and so by completeness of $\{f_k\}$ we see that $Q = 0$, a contradiction.

	We may show similarly that the algebraic multiplicity of each eigenvalue is one, since if $\opnm{dim}\opnm{Ran} P_k > 1$ and $g \in \opnm{Ran} P_k \backslash \{0\}$ is orthogonal to $f_k$, the continuous linear functional $w \mapsto \langle P_k w, g\rangle$ is zero on the linear span of $\{f_k\}_{k=0}^\infty$ which is dense in $L^2(\R)$. Therefore $g = 0$, a contradiction.
\end{proof}

\subsection{Computation of spectral projection norms}\label{ss2Comp}

We now compute the norms of the spectral projections of $L$.  It is obvious that $\langle f_j, g_k\rangle = \delta_{jk}$, using the Kronecker delta. By Lemma \ref{l2Complete} the spectral projection $P_k$ for $L$ defined in \eqref{e2LDef} associated with $\lambda_k$ has the representation
\begin{equation}\label{e2ProjAsIP}
	P_k u = \langle u, g_k\rangle f_k,
\end{equation}
and so
\begin{equation}\label{e2ProjAsNorms}
	\|P_k\| = \|f_k\| \|g_k\|.
\end{equation}

\begin{theorem}\label{t2}
	Let $P_k$ denote the spectral projection for $L$, defined in \eqref{e2LDef} with $T$ from \eqref{e2TDef} and $B$ from \eqref{e2BDef}, for the eigenvalue $\lambda_k = 2k +1+a^2$.  Then we have the asymptotic formula
	\begin{equation}\label{e2Precise}
	 	\|P_k\| = \frac{1}{2(2k)^{1/4}\sqrt{|a|\pi}}\exp\left(2^{3/2}|a|\sqrt{k}\right)\left(1+\BigO(k^{-1/2})\right), \quad k \rightarrow \infty,
	\end{equation}
	from which
	\[
	 	\lim_{k\rightarrow\infty} \frac{1}{\sqrt{k}}\log\|P_k\| = 2^{3/2}|a|.
	\]
	All these statements are also true for $\tilde L = T+\tilde{B}$, with $\tilde{B}$ defined in \eqref{e2tildeBDef}.
\end{theorem}

\begin{remark}\label{r2Rea}
 	We may perform a very similar analysis when replacing $a \in \R$ with $a+ib \in \C$, with the sole change being that the eigenfunctions are 
 	\[
 		f_k(x) = h_k(x + i(a + ib)), \quad g_k(x) = h_k(x - i(a+ib))
 	\]
 	and the eigenvalues are $\{2k+ 1+ (a+ib)^2, k\geq 0\}$.  Since $\|e^{(a+ib)x}h_k(x)\|_{L^2} = \|e^{ax}h_k(x)\|_{L^2}$, we have the same asymptotics for the spectral projection growth at eigenvalues $1+2k+(a+ib)^2$ depending only on the real part, $a$.
 	
 	As discussed in greater generality in Section \ref{secNonOdd}, this also may be seen directly from conjugation on the Fourier transform side with a unitary multiplication operator, $e^{ib\xi}$.  That is, if we had
 	\[
 		\tilde{L} = -\frac{d^2}{d\xi^2} + \xi^2 + 2(a+ib) \frac{d}{d\xi},
 	\]
 	then, for any $u \in L^2$,
 	\[
 		e^{ib\xi}\tilde{L}e^{-ib\xi}u(\xi) = -\frac{d^2}{d\xi^2} u(\xi) + \xi^2 u(\xi) + 2a\frac{d}{d\xi} u(\xi) + 2iab - b^2.
 	\]
 	The aforementioned asymptotics (depending only on $a$) follow from Theorem \ref{t2} and the observation that $(a+ib)^2 = a^2 + 2iab - b^2$, confirming the identification of the spectrum.
\end{remark}

\begin{proof}

From Lemma \ref{l2Eigenfunctions} and \eqref{e2ProjAsNorms}, we have that
\[
 	\|P_k\| = \int_{-\infty}^\infty e^{2ax} h_k(x)^2\,dx.
\]
One measure of correlation between Hermite polynomials is given in \cite[Formula 7.374.7]{GraRyz2000}: for $n \geq m$,
\[
 	\int_{-\infty}^\infty e^{-(x-y)^2}H_m(x) H_n(x)\,dx = 2^n \sqrt{\pi} m! y^{n-m}L^{n-m}_m(-2y^2).
\]
Here, $L^n_m$ are the Laguerre polynomials \cite[Chap.\ V]{Sze1959}. 
  This effectively tests the portion of $H_m(x)H_n(x)$ localized, by a Gaussian, to a neighborhood of $y$.  We choose $n = m = k$ and note that
\begin{equation}\label{e2HermiteToLaguerre}
 	\int_{-\infty}^\infty e^{2ax}h_k(x)^2\,dx = \frac{e^{a^2}}{2^k k! \sqrt{\pi}}\int_{-\infty}^\infty e^{-(x-a)^2}H_k(x)^2\,dx = e^{a^2}L^0_k(-2a^2).
\end{equation}
Asymptotics for the Laguerre polynomials
may be found in \cite[Thm. 8.22.3]{Sze1959}, 
\begin{equation}\label{e2Laguerre}
 	\quad \forall x \in \C\backslash \R_+, \quad L^0_k(x) = \frac{e^{x/2}e^{2(-kx)^{1/2}}}{2\sqrt{\pi}(-x)^{1/4} k^{1/4}}(1+\BigO(k^{-1/2})),
\end{equation}
where the bound may be taken uniform on compact subsets of $\C\backslash \R_+$.  The conclusion \eqref{e2Precise} follows immediately from \eqref{e2HermiteToLaguerre} and \eqref{e2Laguerre}. 
\end{proof}

\section{Perturbations of Schr\"odinger operators with single-well potentials}\label{s3Monomial}

In Section \ref{sec.ho} we relied on readily available asymptotics of Hermite and Laguerre functions. In a more general setting, with perturbations more complicated than multiplication by $x$, we need to enter into the details of finding asymptotics for eigenfunctions of Schr\"odinger operators following Titchmarsh \cite{Tit1954, Titchmarsh-1958-book2}.

\subsection{Definition of operators and basic facts}

We now broaden our class of Schr\"odinger operators with potentials $q$ and introduce perturbations determined by a function $p$. 
\begin{asm} \label{asm:q}
	Let $q \geq 0$ be a real-valued, non-negative, even function satisfying
	\begin{enumerate}[{\upshape i)}]
		\item\label{it.asm.1.i} $q \in C^2(\R) \cap C^3(\R\setminus [-a,a])$ for some $a \geq 0$, 
		\item $q''(x) > 0$ for $x \in (0,\infty)$,
		\item\label{it.asm.1.iii} as $x \to + \infty$
	\begin{equation}\label{q.asym}
	\frac{q'(x)}{q(x)} = \BigO\Big(\frac1x \Big), 
	\quad 
	\frac{q''(x)}{q'(x)}= \BigO\Big(\frac1x \Big), 
	\quad
	\frac{q''' (x)}{q''(x)} = \BigO\Big(\frac1x \Big),
	\end{equation}
		\item\label{it.asm.1.iv} $q^{-1/2} \in L^1((a, \infty))$ for some $a \geq 0$, and
		\item there exists some $\delta > 0$ and $a \geq 0$ for which $q(x) \geq \delta x^2$ for all $x \geq a$.
	\end{enumerate}
\end{asm}
\medskip 
Note that we implicitly assume that the potential $q$ is of single-well type: since $q$ is even, $q'(0) = 0$, and the assumption that $q''(x) > 0$ for all $x > 0$ implies that $q'(x) > 0$ when $x > 0$ as well. We remark also that under reasonable additional hypotheses such as $q(0) = 0$, the first and second parts of \eqref{q.asym} follow from the third.

Define the Schr\"odinger operator
\begin{equation}\label{e3TDef}
\begin{aligned}
T &= -\frac{d^2}{dx^2} + q, \quad 
\Dom(T) = \{ f \in W^{2,2}(\R): q f \in L^2(\R) \},
\end{aligned}
\end{equation}
where $q$ will denote either a function on $\R$ or the corresponding multiplication operator.

Standard examples of $T$ to which our results apply are anharmonic oscillators with $q(x)=|x|^\beta$, $\beta >2$.
The operator $T$ is self-adjoint, bounded from below and has a compact resolvent. The self-adjointness of $T$, with the domain given in \eqref{e3TDef}, follows from the essential self-adjointness of $T$ equipped with the domain $C_0^{\infty}(\R)$, see for instance \cite[Thm.\ X.28]{Reed2}, and the graph norm inequality in Lemma \ref{lem.beta.grn} below. For compactness of resolvent, see for instance \cite[Thm.\ XII.67]{Reed4}.

The spectrum of $T$ is discrete and contains only simple eigenvalues, so we write
\begin{equation}\label{e3TEigs}
\opnm{Spec} (T) = \{ \lambda_0 < \lambda_1 < \lambda_2 < \dots\} \subset \R_+.
\end{equation}
The simplicity of eigenvalues can be shown using the Wronskian: if $f, g \in \Dom(T)$ are eigenfunctions of $T$ corresponding to the same eigenvalue, their Wronskian $f'g - g'f$ must be constant. By the Cauchy-Schwarz inequality we can see that the Wronskian must be integrable and is therefore zero, meaning that $f$ and $g$ are linearly dependent.  Note also that our assumption that $q(x)$ is even means that $Tf(x) = \lambda f(x)$ if and only if $Tf(-x) = \lambda f(-x)$, so by simplicity of eigenvalues $f(x)$ and $f(-x)$ are linearly dependent. This is only possible if $f$ is an odd or an even function.

Given $\lambda \in q(\Bbb{R})$, we define $x_\lambda$, the positive turning point
\begin{equation}\label{xlambda.def}
q(x_\lambda) = \lambda,\quad  x_\lambda > 0.
\end{equation}
From \cite[Eq.\ (1.2)]{Tit1954}, we know that eigenvalues of $T$ obey the estimate
\begin{equation}\label{e3TitEigvals}
\int_{-x_{\lambda_k}}^{x_{\lambda_k}} (\lambda_k - q(x))^{1/2}\,dx = \left( k+\frac{1}{2}\right) \pi + o(1).
\end{equation}
For the special choice of $q(x)=|x|^\beta$, it follows that
\begin{equation}\label{e3klambdak}
\lim_{k\rightarrow\infty} k^{-1}\lambda_k^{\frac{2+\beta}{2\beta}} = \frac{\pi}{2\Omega_\beta}, \quad \mbox{ with} \quad 
\Omega_\beta = \int_0^1 (1-x^\beta)^{1/2}\,dx.
\end{equation}

Lemma \ref{lem:titch} in the Appendix gives solutions $\{y_k\}_{k=0}^\infty$, on the interval $(0,\infty)$, to the equations
\begin{equation}
Ty_k = \lambda_k y_k
\end{equation}
of the form
\begin{equation}\label{yu.rel}
y_k(x) = u_k(x) (1+ \BigO(x_{\lambda_k}^{-1} \lambda_k^{-1/2})), \quad x >0,
\end{equation}
where the function $u_k$ is defined in \eqref{u.def} in the Appendix. By a Wronskian argument, these coincide with the eigenfunctions: Assumption \ref{asm:q} gives that there exists $C_0 > 0$ for which $q'(x) \leq C_0q(x)/x$ for all $x > C_0$, and by Gr\"onwall's inequality, for all $x > C_0$,
\begin{equation}
	 q(x) \leq q(C_0)\exp\left(\int_{C_0}^x \frac{C_0}{t}\,dt\right) \leq q(C_0)e^{C_0\ln x}.
\end{equation}
Therefore there exists some $C \geq 2$ where
\begin{equation}\label{q.poly.bounds}
	\frac{1}{C}x^2 \leq q(x) \leq Cx^C, \quad x > C,
\end{equation}
so $y_k$ is integrable on $(0, \infty)$ by Lemma \ref{lem:u.basic}. Writing
\begin{equation}
	y_k'(x) = -\int_x^\infty y_k''(t)\,dt = -\int_x^\infty (\lambda_k - q(t))y_k(t)\,dt,
\end{equation}
the same estimates in Lemma \ref{lem:u.basic} give that $y_k'$ is integrable as well.  The same argument which provides the simplicity of eigenvalues above then shows that, when $Tv = \lambda_k v$, then $v$ must be a multiple of $y_k$ on $(0, \infty)$.

We will therefore choose the eigenfunctions $\{y_k\}_{k=0}^\infty$ in the form \eqref{yu.rel}; for $x < 0$, we will simply rely on the fact that the eigenfunctions are either odd or even.

In order to compute asymptotics of multiples of the eigenfunctions $y_k$, we present asymptotics for the functions $u_k$ in the Appendix.  Roughly, the eigenfunctions are exponentially small for $|x| \geq x_{\lambda_k}$ and are rapidly oscillating with amplitude $(\lambda_k-q(x))^{-1/4}$ when $|x| \leq x_{\lambda_k}$. (In a small neighborhood of the turning points $x_{\lambda_k}$, we only give rough upper and lower bounds of the $L^2$ mass of the $y_k$.)  This parallels the Wentzel--Kramers--Brillouin (WKB) approximation where the $y_k$ would be approximated by
\[
	a(x) e^{i\zeta(x)}
\]
with $\zeta(x) = \zeta(x,\lambda_k) = \int_x^{x_\lambda} \sqrt{\lambda_k-q(s)}\,ds$ solving the eikonal equation
\[
	(\zeta'(x))^2 + q(x)-\lambda_k = 0
\]
and $a(x) = a(x,\lambda_k) = (\lambda_k - q(x))^{-1/4}$ solving the transport equation
\[
	\zeta''(x) a(x) - 2\zeta'(x)a'(x) = 0.
\]
Our situation is essentially different, however, in that we are not studying a semiclassical Schr\"odinger operator $h^2 D^2 + q(x)$ as the small parameter $h$ tends to zero.

\subsection{Definition of perturbed operators}

We look for perturbations of $T$ which give nice properties like those exploited in Section \ref{sec.ho}.  We wish to have
\begin{equation}\label{e3fkDef}
f_k(x) = e^{p(x)}y_k(x), \qquad g_k(x) = e^{-p(x)}y_k(x)
\end{equation}
as eigenfunctions of our perturbed operator and its adjoint. This leads us to the following formal conjugation of $T$ with a multiplication operator $e^{p(x)}$:
\begin{equation}\label{e3Conjugation}
e^{p(x)}Te^{-p(x)} = T - e^{p(x)}\left[\frac{d^2}{dx^2}, e^{-p(x)}\right] = T + p''(x) - (p'(x))^2 + 2p'(x)\frac{d}{dx}.
\end{equation}
In order to have good properties of the perturbed operator and asymptotics for its eigenfunctions, we impose the following conditions on $p$.
\begin{asm}\label{asm:p}
	Let $p \in C^2(\R)$ be an odd function which is non-negative and increasing on $(0, \infty)$.  Assume furthermore that $\lim_{x \to \infty} p(x) = \infty$ and that, as $x \to \infty$, $p$ obeys the estimates
	\begin{equation}\label{p.bound}
	p(x) = o(x q(x)^\frac12), \quad   
	\frac{p'(x)}{p(x)} = \BigO\Big(\frac1x \Big), \quad   
	|p''(x)| = o(q(x)^\frac12).
	\end{equation}
\end{asm}

Thus the new operator in \eqref{e3Conjugation} is a sum of $T$ and the first order differential operator
\begin{equation}\label{e3BDef}
\begin{aligned}
B = B(p) = 2p'(x) \frac{d}{dx} + p''(x) - (p'(x))^2, \qquad
\Dom(B)  = \Dom(T).
\end{aligned}
\end{equation}

We show that $B$ is a relatively bounded perturbation of $T$ with $T$-bound $0$ (see Lemma \ref{lem.beta.sub}), and we define the perturbed operator as
\begin{equation}\label{e3LDef}
\begin{aligned}
L = T+B, \qquad \Dom(L)  = \Dom(T).
\end{aligned}	
\end{equation}
Using the standard perturbation results \cite[Thm.\ IV.1.1, IV.1.16]{Kato-1966}, it follows that $L$ is closed with compact resolvent.
Consequently, the operator family
\begin{equation}\label{Ls.fam}
\begin{aligned}
L_t = T + B(tp) = T+tB(p) + t(1-t)(p'(x))^2,\qquad \Dom(L_t) = \Dom(T)
\end{aligned}
\end{equation}
is a holomorphic family of type (A); see \cite[Sec.\ VII.2]{Kato-1966} and in particular \cite[Thm.\ VII.2.6]{Kato-1966}. Notice also that $L=L_1$ and $L^*=L_{-1}$. 

We give and use analogues of Lemmas \ref{lem.ho.grn} and \ref{lem.ho.sub}. Notice that the inequality 	$\|Tf\|^2 + \|f\|^2 \leq 2 (\|f''\|^2 + \| q f\|^2 + \|f\|^2) $ for all $f \in \Dom(T)$ is straightforward from the Cauchy-Schwarz inequality.

\begin{lemma}\label{lem.beta.grn}
Let $q$ satisfy parts \ref{it.asm.1.i}-\ref{it.asm.1.iii} of Assumption \ref{asm:q}, and let $T$ be the operator defined in \eqref{e3TDef}. Then, for every $\delta>0$, there exists $k(\delta)>0$ such that for all $f \in \Dom(T)$
	\begin{equation}\label{beta.grn.eq}
	\begin{aligned}
	\|f''\|^2 + (1-\delta)  \| q f\|^2 
	- k(\delta) \|f\|^2 
	&\leq 
	\|Tf\|^2 + \|f\|^2.
	\end{aligned}
	\end{equation}
\end{lemma}
\begin{proof}
For $f \in \Dom(T)$,
\begin{align*}
\|Tf\|^2 &= \|f''\|^2 + \|q f\|^2 - 2\Re\langle f'', q f\rangle
\\ &= \|f''\|^2 + \|q f\|^2 +2 \Re \left(\|q^{1/2}f'\|^2 +\langle f', q'f\rangle\right).
\end{align*}
Then, again using integration by parts
\[
\langle f', q'f\rangle = -\|(q'')^{1/2}f\|^2 - \langle f, q' f'\rangle,
\]
we get that
\[
2\Re \langle f', q'f\rangle = -\|(q'')^{1/2}f\|^2,
\]
so
\begin{equation}\label{e3TfNorm}
\|Tf\|^2 = \|f''\|^2 + \|qf\|^2 + 2\|q^{1/2}f'\|^2 - \|(q'')^{1/2}f\|^2.
\end{equation}

Since $q$ satisfies \eqref{q.asym}, we obtain for a sufficiently large $N \geq 1$ 
\begin{equation}
\begin{aligned}
\|(q'')^{1/2}f\|^2 & \leq M_1 \|f\|^2 +  \int_{\R \setminus [-N,N]} \frac{q''(x)}{q'(x)} \frac{q'(x)}{q(x)} q(x) |f(x)|^2 \, dx 
\\
& \leq M_1 \|f\|^2 + M_2 \int_{\R \setminus [-1,1]} \frac{q(x)}{x^2} |f(x)|^2 \, dx
\\
& \leq M_1 \|f\|^2 + M_2 \int_{\R} q(x)|f(x)|\:|f(x)| \, dx, 
\end{aligned}
\end{equation}
where $M_1$, $M_2>0$. Thus the Cauchy-Schwarz inequality gives for any $\delta>0$
\begin{equation}\label{e3TfNorm2}
\begin{aligned}
\|(q'')^{1/2}f\|^2 \leq \delta \|qf\|^2 + \left(M_1  + \frac{M_2^2}{2 \delta} \right)  \|f\|^2.
\end{aligned}
\end{equation}
Hence we deduce the inequality \eqref{beta.grn.eq} by combining \eqref{e3TfNorm} and \eqref{e3TfNorm2}.
\end{proof}

\begin{lemma}\label{lem.beta.sub}
Let $q$ satisfy parts \ref{it.asm.1.i}-\ref{it.asm.1.iii} of Assumption \ref{asm:q} and let $p$ satisfy Assumption \ref{asm:p}, and let $T$ and $B$ be the operators defined in \eqref{e3TDef} and \eqref{e3BDef}. Then $B$ is relatively bounded with respect to $T$ with the $T$-bound $0$, i.e. for any $\varepsilon>0$ there exists $C(\varepsilon)>0$ such that
\begin{equation}\label{b.beta.sub.op}
\forall f \in \Dom(T), \quad \|B f\| 
\leq
\varepsilon \|T f\| + C(\varepsilon)\|f\|.
\end{equation}
\end{lemma}

\begin{proof}
It is sufficient to derive \eqref{b.beta.sub.op} for $f \in C_0^{\infty}(\R)$, since this set is a core for $T$; see for instance \cite[Thm.\ X.28]{Reed2}.

To estimate $\|Bf\|$, we begin with the triangle inequality:
\begin{equation}\label{e3BfBound1}
	\|Bf\| \leq 2\|p'f'\| + \|p'' f\| + \|(p')^2f\|.
\end{equation}
We wish to separate multiplication and differentiation in the mixed term $\|p'f'\|$, so we write
\begin{equation}\label{e3BfBound2}
	\begin{aligned}
		\|p'f'\| &= \left(-\langle 2p'p'' f, f'\rangle -\langle (p')^2 f, f''\rangle\right)^{1/2} 
		\\ & \leq \|2p'p''f\|^{1/2} \|f'\|^{1/2} + \|(p')^2 f\|^{1/2} \|f''\|^{1/2}.
	\end{aligned}
\end{equation}
From our Assumption \ref{asm:p} and the Cauchy-Schwarz inequality, it is then straightforward to bound the four quantities
\begin{equation}\label{prb.est}
\begin{aligned}
\|(p')^2 f\| & \leq \tilde \varepsilon \|qf\| + C(\tilde \varepsilon) \|f\|,
\\
\|p''f\| & \leq \tilde \varepsilon \|qf\| + C(\tilde \varepsilon) \|f\|,
\\
\|p'p''f\|^{1/2}\|f'\|^{1/2} & \leq (\tilde \varepsilon \|qf\| + C_1(\tilde \varepsilon) \|f\|)^{1/2} (\tilde \varepsilon \|f''\| + C_2(\tilde \varepsilon) \|f\|)^{1/2}
\\ & \leq \tilde \varepsilon \|f''\| + \tilde \varepsilon \|qf\| +  C(\tilde \varepsilon) \|f\|,
\\
\|(p')^2f\|^{1/2}\|f''\|^{1/2} & \leq (\tilde \varepsilon \|qf\| +  C_3(\tilde \varepsilon) \|f\|)^{1/2}\|f''\|^{1/2}
\\
& \leq \tilde \varepsilon \|f''\| + \tilde \varepsilon \|qf\| +  C(\tilde \varepsilon) \|f\|,
\end{aligned}
\end{equation}
with arbitrary $\tilde \varepsilon >0 $ and some $C(\tilde \varepsilon), C_i(\tilde \varepsilon)>0$, $i=1,2,3$. Hence the combination of \eqref{beta.grn.eq}, \eqref{prb.est} and the Cauchy-Schwarz inequality yields \eqref{b.beta.sub.op}.
\end{proof}

\begin{remark}\label{rem:s.sub}
In the special case of $q(x)=|x|^\beta$, $\beta \geq 2$, and 
\begin{equation}\label{e3pBound}
|p^{(k)}(x)| \leq C (1+x^2)^{\frac{\alpha-k}{2}}, \quad k=0,1,2, 
\quad \alpha < \beta/2+1,
\end{equation}
we obtain that $B$ is $s$-subordinated to $T+1$ for
\begin{equation}\label{e3sDef}
s=\frac{1}{2} + \max\left\{\frac{\alpha-1}{\beta},0\right\} < 1,
\end{equation}
that is to say, there exists a $C > 0$ such that
\begin{equation}\label{b.beta.sub}
\forall f \in \Dom(T), \quad \|B f\| 
\leq
C \|(T+1)f\|^s \|f\|^{1-s}.
\end{equation}
This can be seen by estimating the four quantities on the left-hand side of \eqref{prb.est} using a simple consequence of H\"older's inequality
\begin{equation*}
\|x^\delta f\| \leq \|x^\beta f\|^{\delta/\beta}\|f\|^{1-\delta/\beta}
\end{equation*}
for $0 \leq \delta \leq \beta$, its analogue applied on the Fourier transform side and the estimate of the graph norm \eqref{beta.grn.eq}.
\end{remark}

\subsection{Characterization of eigenfunctions}

\begin{lemma}\label{l3Eigenfunctions}
Let $q$ satisfy parts \ref{it.asm.1.i}-\ref{it.asm.1.iv} of Assumption \ref{asm:q} and let $p$ satisfy Assumption \ref{asm:p}, and let $L$ be the operator in \eqref{e3LDef}. Then the functions $f_k$ and $g_k$, defined in \eqref{e3fkDef}, belong to $\Dom(L) = \Dom(L^*) = \Dom(T)$ and satisfy the eigenfunction relations
\begin{equation}\label{e3fkEigs}
\begin{aligned}
Lf_k = \lambda_k f_k, \\
L^*g_k = \lambda_k g_k.
\end{aligned}
\end{equation}
Moreover, 
	\[
	\opnm{Spec}(L) = \opnm{Spec}(L^*) = \opnm{Spec}(T),
	\]
	All eigenvalues of $L$ and $L^*$ have algebraic multiplicity one; therefore $\{f_k\}_{k=0}^\infty$ forms the set of all eigenvectors of $L$ and $\{g_k\}_{k=0}^\infty$ forms the set of all eigenvectors of $L^*$.
	
\end{lemma}

\begin{proof}
	We follow Theorem \ref{tAOlver}, which is reproduced from \cite[Chap.\ 6, Thm.\ 2.1]{Olv1997}, in describing the solutions of the ordinary differential equation of second order
	\[
	T w(x)  = \mu w(x).
	\]
	In the notation of Theorem \ref{tAOlver}, we use $f(x) = q(x)$ and $g(x) = \mu$ to arrive at the error-control function
	\begin{align*}
	F(x) &= \int^x\left(q(t)^{-1/4}\frac{d^2}{dt^2}q(t)^{-1/4} - \mu q(t)^{-1/2}\right)\,dt 
	\\ &= \int^x\left(
	\frac{5}{16} \frac{q'(t)^2}{q(t)^{5/2}}
	- \frac{1}{4} \frac{q''(t)}{q(t)^{3/2}} 
	- \frac{\mu}{q(t)^{1/2}} \right)\,dt.
	\end{align*}
	We then have two linearly independent solutions given as
	\begin{equation}\label{e3wpmDef}
	\forall x \geq R, \quad w_{\pm}(x) = q(x)^{-1/4} \exp\left(\pm\int q(x)^{1/2}\,dx\right)\left(1+r_{\pm}(x)\right).
	\end{equation}
	The remainders are bounded by a quantity involving the variation of $F$ on $[R,\infty)$:
	\begin{equation}\label{r.est}
	\begin{aligned}
		|r_\pm(x)| 
	& \leq 
	\exp\left(\frac{1}{2}\mathcal{V}_{R,x}(F)\right)-1 
	 =
	\exp\left(\frac{1}{2}\int_R^{\infty} |F'(t)|\,dt\right)-1,
	\\
	|r_\pm'(x)| 
	& \leq 2 q(x)^{1/2} \left(
	\exp\left(\frac{1}{2}\int_R^{\infty} |F'(t)|\,dt\right)-1 \right).
	\end{aligned}
		\end{equation}
	From Assumption \ref{asm:q}, $F'$ is easily seen to be integrable on 
	$[R, \infty)$ when $R > 0$ is sufficiently large. If we select $R$ so large that
	\begin{equation*}
	\int_R^\infty |F'(t)| d t \leq \frac{1}{2},
	\end{equation*}
	then 
	\begin{equation}\label{rr'.est}
	|r_\pm(x)| \leq \frac{1}{2}, \qquad |r_\pm'(x)| \leq q(x)^{1/2}.
	\end{equation}
	
	A similar analysis shows that, on $(-\infty, -R]$, one has two $C^2$ solutions $W_{\pm}$ satisfying 
	\[
	T W_\pm(x) = \mu W_\pm(x), \quad x \leq -R,
	\]
	with a similar expression and error bounds on $(-\infty, -R]$ as $R \rightarrow \infty$.
	
	We now turn to a solution of the differential equation
	\begin{equation}\label{L.ode}
	Lf(x) = \mu f(x).
	\end{equation}
	It is evident that $y = e^{-p}f$ solves an ordinary differential equation
	\begin{equation}\label{T.ode}
	Ty(x) = \mu y(x).
	\end{equation}
	Similarly, if $y$ solves \eqref{T.ode}, then $f=e^p y$ solves \eqref{L.ode}.
	
	For $R > 0$, any solution $y$ of \eqref{T.ode} satisfies
	\begin{equation}\label{e3apmwpm}
	y(x) = b_+ w_+(x) + b_- w_-(x), \quad x \geq R,
	\end{equation}
	for $b_\pm \in \C$ and $w_\pm$ from \eqref{e3wpmDef} and similarly in $(-\infty,-R]$.
	Thanks to \eqref{e3wpmDef} we see that
	\begin{equation}\label{epw.est.1}
	\begin{aligned}
	e^{p(x)}w_{\pm}(x) & =  q(x)^{-1/4} \exp\left(\pm \int_R^{x} q(t)^{1/2} 
	\left(
	1+ \frac{p'(t)}{q(t)^{1/2}}
	\right)
	\,dt 
	\mp q(R)^{1/2}\right)
	\\
	& \quad \times \left(1+r_{\pm}(x)\right)
	\end{aligned}
	\end{equation}
    and $p'(t)q(t)^{-1/2} = (p'(t)/p(t))(p(t)/q(t)^{-1/2}) = o(1)$ as $t \to \infty$ by Assumption \ref{asm:p} and \eqref{p.bound}.
	Moreover using  \eqref{q.asym}, we obtain (with some $c>0$)
	\begin{equation}\label{epw.est.2}
	\int_R^{x} q(t)^{1/2} \, dt 
	\geq 
	\frac 1{q'(x)} \int_R^{x} q'(t)q(t)^{1/2} \, dt
	\geq 
	 c \, x q(x)^{1/2},
	\end{equation}
	thus  $e^p w_+ \notin L^2((R,\infty))$ and $e^p w_- \in L^2((R,\infty))$, see Assumption \ref{asm:q}. Analogous claims hold also for $e^p W_\pm$ in $(-\infty,-R]$. 
	Hence we have from \eqref{e3wpmDef}--\eqref{epw.est.2} that if $y=e^{-p}f$ and $y$ solves \eqref{T.ode}, then $y \in L^2(\R)$ if and only if $f \in L^2(\R)$. Moreover by similar reasoning, $qy \in L^2(\R)$ if and only if $qf \in L^2(\R)$ and, using estimates for $|r_\pm'|$ in \eqref{r.est}--\eqref{rr'.est}, $y' \in L^2(\R)$ if and only if $f'\in L^2(\R)$. In summary, $y$ solves \eqref{T.ode} and $y \in \Dom(T)$ if and only if $f=e^{-p} y$ solves \eqref{L.ode} and $f \in \Dom(L) = \Dom(T)$. Thus $\opnm{Spec}(T) = \opnm{Spec}(L)$.
	
    Finally, we consider the family $L_t$, $t \in [-1,1]$, as in \eqref{Ls.fam}.
	The reasoning above gives that all $L_t$ have identical spectrum, in particular $\opnm{Spec}(L^*) = \opnm{Spec}(L_{-1}) = \opnm{Spec}(T)$. 
	In order to obtain agreement of the algebraic multiplicities, defined as the dimensions of the spaces of root vectors, fix an eigenvalue $\lambda \in \opnm{Spec} (T)$ and an $\eps > 0$ smaller than the distance from $\lambda$ to any other eigenvalue. Using \cite[Thm.\ VII.1.3]{Kato-1966}, we have that the trace of the spectral projection,
	\[
	\opnm{Tr}\left(\frac{1}{2\pi i}\int_{|\zeta - \lambda| = \eps} (\zeta - L_t)^{-1}\,d\zeta\right),
	\]
	is a continuous function of $t$.  This trace gives the dimension of the space of root vectors associated with the eigenvalue $\lambda$, and therefore the algebraic multiplicity of each eigenvalue matches that of $T = L_0$, which is one.
\end{proof}

\begin{remark}\label{rem.L1}
The growth condition in Assumption \ref{asm:q}, \ref{it.asm.1.iv} is important for analyzing a more general class of potentials. It allows to use simple asymptotic approximations for solutions of \eqref{T.ode}, see Theorem \ref{tAOlver} in Appendix. In the special case $q(x) = x^2$, we may obtain similar asymptotics due to the very sharp asymptotics readily available for the Hermite functions.  In particular, the characterization \eqref{e2hkDef} allows us to establish the claim of Lemma \ref{l3Eigenfunctions} for $q(x) = x^2$ when $p$ satisfies Assumption \ref{asm:p}. Moreover, observe that without part \ref{it.asm.1.iv} of Assumption \ref{asm:q}, Lemma \ref{lem:titch} and Proposition \ref{prop:u.exp} justify that if $y$ is an eigenfunction of $T$, then $f=e^py$ satisfies \eqref{L.ode}, $f \in L^2(\R)$ and $qf \in L^2(\R)$; however, to conclude that $f \in \Dom(L)$, a suitable estimate on $y'$ is needed.
\end{remark}

\begin{lemma}\label{l3Complete}
Let $q$ satisfy Assumption \ref{asm:q}, let $p$ satisfy Assumption \ref{asm:p}, and let $L$ be the perturbed operator in \eqref{e3LDef}. Then the sequences of functions $\{f_k\}_{k = 0}^\infty$ and $\{g_k\}_{k=0}^\infty$ defined in \eqref{e3fkDef} are each complete in $L^2(\R)$.
\end{lemma}

\begin{proof}
	We need only consider the claim about $L$ and the $\{f_k\}_k$. We have already showed that all eigenvalues of $L$ are simple and $\{f_k\}_k$ are all eigenvectors. To show that the eigensystem of $L$ is complete, we use \cite[Thm.\ X.3.1]{GohGolKaa1990}. Thus we have to verify that, for some $z_0 \in \R$, the resolvent $(L-z_0)^{-1}$ belongs to the Schatten class $\mathfrak{S}_r$ with $r \geq 1$ and that $\Num(L - z_0)$ lies in a closed sector with vertex at zero and opening $\frac \pi r$.
	
	To this end, we take $z_0 <0$ and $|z_0|$ sufficiently large and estimate $s$-numbers of $(L-z_0)^{-1}$. By \cite[Prop.\ VI.1.3]{GohGolKaa1990} with the representation of the resolvent
	$$ 
	(L-z_0)^{-1} = (T-z_0)^{-1} (I + B(T-z_0)^{-1} )^{-1},
	$$
	one arrives at the estimate (with some $C>0$)
	\begin{equation*}
	s_j((L- z_0)^{-1}) \leq C s_j((T-z_0)^{-1}).
	\end{equation*}
	Since $T$ is positive definite, the $s$-numbers of $(T-z_0)^{-1}$ coincide with the eigenvalues of $(T-z_0)^{-1}$. Since $q(x) \geq \delta x^2$ for $x$ sufficiently large, we get from the min-max principle, see e.g.\ \cite[Thm.\ XIII.1 and Problem\ 1]{Reed4}, that $\lambda_j \geq \delta^{-1/2} (2j+1) + C$, $j \in \N_0$, hence 
	$(L-z_0)^{-1} \in \mathfrak{S}_r$ for any $r>1$.
	
	Secondly, we prove that $\Num(L-z_0)$ is contained in a sector whose opening can be made arbitrarily small while keeping the vertex at zero (by taking $|z_0|$ large). 
	We take $f \in \Dom(L)$ and simplify the expression for the quadratic form of $B$ by integration by parts,
	\begin{equation}
	\langle B f, f \rangle 
	= 
	\langle p'' f - p'^2 f + 2 p' f', f \rangle 
	=
	- \|p' f\|^2 + 2 i \Im \langle f', p' f \rangle.
	\end{equation}
	As in \eqref{prb.est}, using Assumption \ref{asm:p} and the Cauchy-Schwarz inequality, we estimate $\Im \langle B f, f \rangle$ and  $\Re  \langle B f, f \rangle$, 
	\begin{equation}\label{Bqf.est}
	\begin{aligned}
	|\Im \langle Bf, f \rangle| 
	&\leq 
	2 |\langle f', p'f \rangle| 
	\leq
	\|f'\| \|p' f \|  \leq \|f'\| (\varepsilon \|q^{1/2} f\| + C_1(\varepsilon) \|f\|) 
	\\  
	&\leq  \varepsilon \left(
	\|f'\|^2 + \|q^{1/2} f\|^2) + C_2(\varepsilon) \|f\|^2
	\right),
	\\
	|\Re \langle Bf, f \rangle| & = \|p'f\|^2 
	\leq \varepsilon \|q^{1/2} f\|^2 + C_3(\varepsilon) \|f\|^2,
	\end{aligned}
	\end{equation}
	where $\varepsilon>0$ is arbitrary and $C_i(\varepsilon)>0$, $i=1,2,3$. Hence using integration by parts again and \eqref{Bqf.est}, we obtain
	\begin{equation}\label{ReL.est}
	\begin{aligned}
	\Re \langle (L -z_0) f, f \rangle & = \|f'\|^2 + \|q^{1/2} f\|^2 - \|p' f\|^2 + |z_0|\|f\|^2
	\\ 
	& 	\geq (1-\varepsilon) \left(
	\|f'\|^2 + \|q^{1/2} f\|^2 + (|z_0|- C_3(\varepsilon))  \|f\|^2 
	\right).
\end{aligned}
	\end{equation}
For any $\varepsilon>0$, we can take $|z_0|$ so large that $C_2(\varepsilon) \leq |z_0| - C_3(\varepsilon)$, hence it follows from \eqref{Bqf.est} and \eqref{ReL.est} that 
\begin{align*}
&	|\Im \langle (L-z_0) f, f \rangle| \leq \frac \varepsilon{1-\varepsilon} \Re \langle (L -z_0)f, f \rangle. \qedhere
\end{align*}
\end{proof}

\begin{remark}\label{rem:ho.lem}
As in Remark \ref{rem.L1}, the additional conditions on $q$ in Lemma \ref{l3Complete}, in particular that $q(x) \geq \delta x^2$, for $x$ sufficiently large, only simplify the proof. This condition was used to estimate eigenvalues of $T$ from below using the min-max principle and get that $(L-z_0)^{-1} \in \mathfrak{S}_r$ with any $r>1$. For other potentials $q$, one could use instead the asymptotic formula \eqref{e3TitEigvals}. Notice in particular that the claims of Lemmas \ref{l3Eigenfunctions} and \ref{l3Complete} are valid for $q(x)=x^2$ and $p$ satisfying Assumption \ref{asm:p}.
\end{remark}

\subsection{Computation of spectral projection norms}
Under the assumptions of Lemma \ref{l3Complete} the sequence $\{f_k\}_{k=0}^\infty$ is complete. Since $\langle f_k, g_j\rangle = \delta_{jk}\|y_k\|^2$, the spectral projection $P_k$ of $L$ associated with $\lambda_k$ has the representation
\begin{equation}\label{Pk.repres}
P_ku = \frac{1}{ \|y_k\|^2}  \langle u , g_k \rangle f_k.
\end{equation}
It is easy to see that its norm satisfies
\begin{equation}\label{Pk.fk.gk}
\|P_k\| = \frac{\|f_k\| \|g_k\|}{\|y_k\|^2}.
\end{equation}

\begin{theorem}\label{t3}
	Let $q$ satisfy Assumption \ref{asm:q}, let $x_\lambda$ be defined by \eqref{xlambda.def}, let $p$ satisfy Assumption \ref{asm:p}, and let $L$ be defined as in \eqref{e3LDef}. Finally, let $P_k$ denote the spectral projection of $L$ at the eigenvalue $\lambda_k$.
	Then
	\begin{equation}\label{e3Conclusion}
	\lim_{k\rightarrow\infty} \frac{\log\|P_k\|}{p(x_{\lambda_k})} = 2. 
	\end{equation}	
\end{theorem}
\begin{proof}
	Within the proof, $C$ denotes a positive, sufficiently large constant that can vary in every step. We allow $C$ to depend on $p$ and $q$ but not on $\eps$; if the constant depends on $\eps$, then we write $C_\eps$.

	Let $\{y_k\}_{k=0}^\infty$ be as in \eqref{yu.rel}. Since $y_k(x)^2$ is always an even function and appealing to \eqref{u.norm},
	\begin{equation}\label{yk.norm}
		\begin{aligned}
		\|y_k\|^2 &= 2\left(1+\BigO(x_{\lambda_k}^{-1}\lambda_k^{-1/2})\right)\int_0^\infty u_k(x)^2\,dx 
		\\ &= 2\left(1+\BigO(x_{\lambda_k}^{-1/3}\lambda_k^{-1/6})\right)\int_0^{x_{\lambda_k}} \frac{dx}{\sqrt{\lambda_k-q(x)}}.
		\end{aligned}
	\end{equation}
	Using also \eqref{e3fkDef} and the fact that $p(x)$ is odd, a change of variables shows that
	\begin{equation}\label{fk.gk.norm}
		\|f_k\|\:\|g_k\| = \int_{-\infty}^\infty e^{2p(x)}y_k(x)^2\,dx = 2 \left(1+\BigO(x_{\lambda_k}^{-1}\lambda_k^{-1/2}\right) \int_0^\infty \cosh(2p(x))u_k(x)^2\,dx.
	\end{equation}
	Using \eqref{Pk.fk.gk}, the first equality in \eqref{yk.norm}, and \eqref{fk.gk.norm}, we obtain
	\begin{equation}\label{Pk.1}
	\log \|P_k\| = \log \int_{0}^\infty \cosh(2p(x)) u_k(x)^2\,dx - \log \int_{0}^\infty u_k(x) ^2\,dx + o(1),
	\end{equation}
	as $k \to \infty$.

	Fix $\varepsilon \in (0,1)$. Denote by $b_k^\pm = q^{-1}((1+\eps)\lambda_k)$ the numbers $b^\pm$ from Lemma \ref{lem:u.basic}.\ref{b.lam} corresponding to $\lambda = \lambda_k$; note that $b_k^{\pm} \to \infty$ as $k \to \infty$.
	Then, by Proposition \ref{prop:u.exp}, \eqref{p.bound}, and the fact that $q'(x)/q(x) = \BigO (1/x)$ as $x \to +
	\infty$, for $k$ sufficiently large depending on $\varepsilon$ we have
	\begin{equation}\label{int.est.inf}
	\begin{aligned}
	\int_{b_k^+}^\infty \cosh(2p(x)) u_k(x)^2\,dx 
	&\leq
	C_\eps \int_{b_k^+}^\infty e^{-\frac{1}{C_\eps}\frac{q(x)^{3/2}}{q'(x)} + 2p(x)-\frac{1}{4}\log q(x)}\,dx
	\\ 
	&\leq 
	C_\eps \int_{b_k^+}^\infty e^{-\frac{1}{C_\eps}xq(x)^{1/2}} \,dx 
	\\
	& 
	\leq 
	C_\eps \int_{b_k^+}^\infty e^{-\frac{1}{C_\eps}\sqrt{\lambda_k} x} \,dx 
	\\
	&
	\leq
	C_\eps e^{-\frac{1}{C_\eps}\sqrt{\lambda_k} x_{\lambda_k}}.
	\end{aligned}
	\end{equation}
	
	We observe that, since $p$ is increasing for $x > 0$,
	\begin{equation}\label{int.cosh.small.x}
	\int_0^{b_k^+} \cosh(2p(x))u_k(x)^2\,dx \leq \cosh(2p(b_k^+))\int_0^\infty u_k(x)^2\,dx.
	\end{equation}
	In order to show that that, as $k \to \infty$,
	\begin{equation}\label{bkp.upper.bound}
		 \int_0^\infty \cosh(2p(x))u_k(x)^2\,dx \leq \left(\frac{1}{2}+o(1)\right)e^{2p(b_k^+)}\int_0^\infty u_k(x)^2\,dx,
	\end{equation}
	it suffices to show that the contribution from $[b_k^+, \infty)$ in \eqref{int.est.inf} is neglible. Because $q'(x)$ is increasing,
	\begin{equation}\label{eq.qprime.max}
		\int_{b_k^-}^{x_{\lambda_k}} (\lambda_k-q(x))^{-1/2}\, dx = \int_{b_k^-}^{x_{\lambda_k}} \frac{q'(x)}{q'(x)\sqrt{\lambda_k-q(x)}}\, dx \geq \frac{1}{q'(x_{\lambda_k})}(2\sqrt{\eps\lambda_k}).
	\end{equation}
	Using that $q(x)/q'(x) \geq x/C$ and \eqref{q.poly.bounds}, we see that
	\begin{equation}\label{eq.qprime.max.2}
		\frac{2\sqrt{\eps\lambda_k}}{q'(x_{\lambda_k})} = \frac{2\sqrt{\eps}q(x_{\lambda_k})}{q'(x_{\lambda_k})\sqrt{q(x_{\lambda_k})}} \geq \frac{1}{C_\eps}x_{\lambda_k}^{-C}.
	\end{equation}
	Since the contribution from $[b_k^+, \infty)$ in \eqref{int.est.inf} is exponentially small in $x_{\lambda_k}$ and since $e^{2p(b_k^+)} \geq 1$, we obtain \eqref{bkp.upper.bound}.

	Inserting this into \eqref{Pk.1} gives, for $k$ sufficiently large depending on $\eps$,
	\begin{equation}
		\log\|P_k\| \leq 2p(b_k^+) + C_\eps.
	\end{equation}
	Knowing that $p(b_k^+) \to \infty$ as $k\to \infty$, we conclude that
	\begin{equation}\label{Pk.limsup}
	\limsup_{k \to \infty} 
	\frac{\log \|P_k\|}{ p(x_{\lambda_k})} 
	\leq 
	2 \limsup_{k \to \infty} 
	\frac{p(b_k^+)}{p(x_{\lambda_k})}. 
	\end{equation}

	Our goal now is to show that, for $\eps > 0$ sufficiently small,
	\begin{equation}\label{pbkp.pxlam}
		\limsup_{k \to \infty} \frac{p(b_k^+)}{p(x_{\lambda_k})} \leq 1 + C\eps.
	\end{equation}
	The derivative of the inverse of $q$ gives
	\[
		\frac{d}{d\lambda}q^{-1}(\lambda) = \frac{1}{q'(q^{-1}(\lambda))}
	\]
	which is decreasing for $\lambda > 0$.  Therefore
	\begin{equation}\label{bp.1}
	b_k^+ - x_{\lambda_k} 
	= 
	q^{-1}((1+\varepsilon)\lambda_k) - q^{-1}(\lambda_k) 
	\leq 
	\frac{\varepsilon \lambda_k}{q'(q^{-1}(\lambda_k))}
	=
	\frac{\varepsilon q(x_{\lambda_k})}{q'(x_{\lambda_k})}.
	\end{equation}
	Since $q'$ is non-decreasing for $x > 0$, we see that $q(x) \leq q(0) + xq'(x)$ for all $x > 0$.  Therefore, for all $k>k_0$ with $k_0$ independent of $\varepsilon$, 
	\begin{equation}\label{bp.2}
	\frac{b_k^+ - x_{\lambda_k}}{x_{\lambda_k}} 
	\leq 
	\frac{\varepsilon q(x_{\lambda_k})}{x_{\lambda_k} q'(x_{\lambda_k})} 
	\leq 
	\varepsilon
	\left(
	1 + \frac{q(0)}{x_{\lambda_k} q'(x_{\lambda_k})}
	\right) 
	\leq 2 \varepsilon.
	\end{equation}
	Using that $p$ is increasing, $p'(x)/p(x) = \BigO(1/x)$ as $x \to \infty$ and the mean value theorem, we obtain for some $\xi \in (x_{\lambda_k},b_k^+)$ that
	\begin{equation}\label{ratio.without.limsups}
	\begin{aligned}
	\frac{p(b_k^+)}{p(x_{\lambda_k})} & = 1 + \frac{p(b_k^+) - p(x_{\lambda_k})}{p(x_{\lambda_k})}
	= 
	1 + \frac{p'(\xi) (b_k^+  - x_{\lambda_k}) }{p(x_{\lambda_k})}
	\\
	&\leq
	1 + C \frac{p(b_k^+)}{p(x_{\lambda_k})} \frac{b_k^+ - x_{\lambda_k}}{x_{\lambda_k}}
	\leq
	1 + 2 C \varepsilon \frac{p(b_k^+)}{p(x_{\lambda_k})}.
	\end{aligned}
	\end{equation}
	Rearranging terms gives \eqref{pbkp.pxlam}, and taking $\eps \to 0$ in \eqref{Pk.limsup} gives
	\begin{equation}\label{Pk.limsup.2}
	\limsup_{k \to \infty} 
	\frac{\log \|P_k\|}{ p(x_{\lambda_k})} 
	\leq 
	2. 
	\end{equation}
	
	In the second step, starting from \eqref{Pk.1}, using \eqref{u.norm} and Corollary \ref{cor:u.bpm}, we obtain
	\begin{equation}\label{Pk.geq.1}
	\begin{aligned}
	\log \|P_k\| & \geq  \log \int_{b_k^-}^{b_k^+} \cosh(2p(x)) u_k(x)^2\,dx - \log \int_{0}^\infty u_k(x) ^2\,dx - C
	\\ 
	& \geq 
	\log \cosh(2p(b_k^-)) - \log \frac{\int_{0}^\infty u_k(x) ^2\,dx}{\int_{b_k^-}^{b_k^+} u_k(x)^2\,dx} - C
	\\ 
	& \geq 
	2p(b_k^-) - \log \frac{\int_{0}^{x_{\lambda_k}} (\lambda_k - q(x))^{-1/2}\, dx}{\int_{b_k^-}^{x_{\lambda_k}} (\lambda_k-q(x))^{-1/2}\, dx} - C
	\end{aligned}
	\end{equation}
	as $k\to \infty$. We aim to estimate the fraction in the second term,
	\begin{equation}\label{Pk.geq.int.fraction}
		\frac{\int_{0}^{x_{\lambda_k}} (\lambda_k - q(x))^{-1/2}\, dx}{\int_{b_k^-}^{x_{\lambda_k}} (\lambda_k-q(x))^{-1/2}\, dx} = 1 + \frac{\int_{0}^{b_k^-} (\lambda_k - q(x))^{-1/2}\, dx}{\int_{b_k^-}^{x_{\lambda_k}} (\lambda_k-q(x))^{-1/2}\, dx}.
	\end{equation}
	Because $q(x)$ is increasing,
	\begin{equation}
		\int_{0}^{b_k^-} (\lambda_k - q(x))^{-1/2}\, dx \leq b_k^- (\lambda_k-q(b_k^-))^{-1/2} = \frac{b_k^-}{\sqrt{\eps\lambda_k}},
	\end{equation}
	and a lower bound on the denominator comes from \eqref{eq.qprime.max}. Using that $\lambda_k = q(x_{\lambda_k})$, that $q'(x)/q(x) \leq C/x$, and that $b_k^- \leq x_{\lambda_k}$, we continue from \eqref{Pk.geq.int.fraction} to obtain
	\begin{equation}
		\frac{\int_{0}^{x_{\lambda_k}} (\lambda_k - q(x))^{-1/2}\, dx}{\int_{b_k^-}^{x_{\lambda_k}} (\lambda_k-q(x))^{-1/2}\, dx} \leq 1+\frac{b_k^- q'(x_{\lambda_k})}{2\eps \lambda_k} \leq 1 + C\frac{b_k^-}{\eps x_{\lambda_k}} \leq 1 + \frac{C}{\eps}.
	\end{equation}
	Along with \eqref{Pk.geq.1}, this gives as $k\to\infty$
	\begin{equation}
		\liminf_{k\to\infty}\frac{\log \|P_k\|}{p(x_{\lambda_k})} \geq \liminf_{k\to\infty}\frac{2p(b_k^-)-C_\eps}{p(b_k^-)} \geq 2\liminf_{k\to\infty}\frac{p(b_k^-)}{p(x_{\lambda_k})}.
	\end{equation}
	
	To complete the proof, we note that the relation between $b_k^+$ and $x_{\lambda_k}$ is similar to the relation between $x_{\lambda_k}$ and $b_k^-$: that is, as $\eps \to 0$,
	\begin{equation}
		x_{\lambda_k} = q^{-1}((1+\delta)q(b_k^-)), \quad \delta = \frac{\eps}{1-\eps} = \BigO(\eps), \quad \textnormal{as }\eps \to 0.
	\end{equation}
	The reciprocal of \eqref{pbkp.pxlam} therefore gives, as $\eps \to 0$,
	\begin{equation}
		\liminf_{k\to\infty} \frac{p(b_k^-)}{p(x_{\lambda_k})} \geq 1-C\eps,
	\end{equation}
	allowing us to conclude that
	\begin{equation}\label{Pk.liminf.2}
	\liminf_{k \to \infty} 
	\frac{\log \|P_k\|}{ p(x_{\lambda_k})} 
	\geq 
	2. 
	\end{equation}
	Taken with \eqref{Pk.limsup.2}, this completes the proof of the theorem.
\end{proof}

Next, we formulate a corollary for $q(x) = |x|^\beta$ and special choices of $p$ to get a range of exponential growth rates for $\|P_k\|$. In this case, if $x_{\lambda_k}$ is as in \eqref{xlambda.def} for $\lambda_k$ the eigenvalues of $T = -\frac{d^2}{dx^2}+q(x)$ and with $\Omega_\beta$ as in \eqref{e3klambdak}, we compute that
\begin{equation}\label{xlambdak.qbeta.asym}
	x_{\lambda_k} = (1+o(1))\left(\frac{\pi}{2\Omega_\beta}k\right)^{\frac{2}{2+\beta}}, \quad \textnormal{as }k\to\infty.
\end{equation}
Where $p(x)$ is only specified for $x$ sufficiently large, $p$ may be extended to an odd smooth increasing function satisfying Assumption \ref{asm:p} via a smooth partition of unity.

\begin{corollary}\label{cor:Pk.beta}
Let $q(x)= |x|^\beta$, $\beta \geq 2$, let $p$ be extended to satisfy Assumption \ref{asm:p}, let $L$ be defined as \eqref{e3LDef}, let $\{P_k\}_k$ be its spectral projections and let $\Omega_\beta$ be as in \eqref{e3klambdak}.
\begin{enumerate}[{\upshape i)}]
\item For any $0 < \sigma < 1$, let $\alpha = \frac{2+\beta}{2}\sigma$ and let
\begin{equation}\label{alpha.beta.cond}
p(x)=\frac 12 x^\alpha, \quad \forall x > 1.
\end{equation}
Then 	
\begin{equation}
\lim_{k \to \infty} \frac{\log \|P_k\|}{k^\sigma} 
= \left(
\frac \pi{2\Omega_\beta}
\right)^\sigma.
\end{equation}
\item For $\gamma > 0$, let
\begin{equation}
p(x) = \frac{x^{\frac{2+\beta}{2}}}{2(\log x)^\gamma}, \quad \forall x > 2.
\end{equation}		
Then
\begin{equation}
\lim_{k \rightarrow \infty} \frac{\log \|P_k\|}{k/(\log k)^\gamma} = \frac{\pi}{2\Omega_\beta}\left(\frac{\beta+2}{2}\right)^\gamma.
\end{equation}
\item\label{it.cor.log} Let
\begin{equation}
p(x) = \frac{1}{2} \log x \quad \forall x > 2.
\end{equation}		
Then
\begin{equation}
\lim_{k \rightarrow \infty} \frac{\log \|P_k\|}{\log k} = \frac{2}{2+\beta}.
\end{equation}
\item\label{it.cor.loglog} Let
\begin{equation}
p(x) = \frac{1}{2} \log \log x \quad \forall x > e^2.
\end{equation}		
Then
\begin{equation}
\lim_{k \rightarrow \infty} \frac{\log \|P_k\|}{\log \log k} = 1.
\end{equation}
\end{enumerate}
\end{corollary}

\begin{proof} As an example, we perform the first computation where $p(x) = \frac{1}{2}x^\alpha$. By \eqref{xlambdak.qbeta.asym}, for $k$ sufficiently large we have
\begin{equation}
p(x_{\lambda_k}) = \frac{1}{2}x_{\lambda_k}^\alpha = \frac{1}{2}\left(\frac{\pi}{2\Omega_\beta}k\right)^{\frac{2}{2+\beta}\alpha}(1+o(1)).
\end{equation}
Since $\frac{2}{2+\beta}\alpha = \sigma$,
\begin{equation}
\frac{\log \|P_k\|}{k^\sigma} = \frac{1}{2}(1+o(1))\left(\frac{\pi}{2\Omega_\beta}\right)^\sigma \frac{\log\|P_k\|}{p(x_{\lambda_k})} \to \left(\frac{\pi}{2\Omega_\beta}\right)^\sigma, \quad \textnormal{as }k \to \infty,
\end{equation}
by Theorem \ref{t3}.

The other computations are similar, and more precise asymptotics for \ref{it.cor.log} and \ref{it.cor.loglog} can be found in Corollary \ref{cor:Pk.beta.slow}.
\end{proof}

\begin{remark}
Notice that we have included $\beta=2$ in Corollary \ref{cor:Pk.beta}. This is because Theorem \ref{t3} can indeed be applied, since the claims of both Lemmas \ref{l3Eigenfunctions} and \ref{l3Complete} can be deduced in this case from asymptotics for the Hermite functions, as discussed in Remark \ref{rem:ho.lem}.
\end{remark}

For rapidly-growing $p$, a large portion of the mass of $e^{p(x)}y_k(x)$ is found near the turning point $x_{\lambda_k}$. Since we do not attempt a detailed analysis of $y_k$ near the turning point $x_\lambda$, we only obtain information on $\log\|P_k\|$ which is approximately $2p(x_{\lambda_k})$.  On the other hand, if $p$ grows slowly, the majority of the mass of $e^{p(x)}y_k(x)$ is spread out over the interval $[-x_{\lambda_k}, x_{\lambda_k}]$, which allows us to show the following more precise asymptotics for $\|P_k\|$. As a corollary for $q(x)=|x|^\beta$, we find sharp polynomial rates of spectral projection growth. 

\begin{theorem}\label{t3.2}
	Let $q$ satisfy Assumption \ref{asm:q}, let $x_\lambda$ be defined by \eqref{xlambda.def}, let $p$ satisfy Assumption \ref{asm:p} and in addition, as $x \to + \infty$,
	\begin{equation}\label{p'.slow}
	 p'(x) = \BigO \Big(\frac 1x \Big),
	\end{equation}
	let $L$ be defined as in \eqref{e3LDef}, and denote by $P_k$ the spectral projection for $L$ at the eigenvalue $\lambda_k$.
	Then
	\begin{equation}\label{e3Conclusion.2}
	\lim_{k\rightarrow\infty} \|P_k\| \frac{\int_0^{x_{\lambda_k}} (\lambda_k - q(x))^{-1/2}\,dx}{\int_0^{x_{\lambda_k}} e^{2p(x)}(\lambda_k - q(x))^{-1/2} \,dx} = \frac12. 
	\end{equation}	
\end{theorem}
\begin{proof}
Following the beginning of the proof of Theorem \ref{t3}, from \eqref{Pk.fk.gk}, \eqref{yk.norm}, and \eqref{fk.gk.norm}, we have
\begin{equation}\label{eq.t3.2.copy}
	\|P_k\| = \frac{\int_0^\infty \cosh(2p(x))u_k(x)^2\,dx}{\int_0^{x_{\lambda_k}}\pi(\lambda_k-q(x))^{-1/2}\,dx}\left(1+o(1)\right), \quad k\to \infty.
\end{equation}
Fix $\eps \in (0,1/3)$ and let $b_k^\pm$ be as in Lemma \ref{lem:u.basic}.\ref{b.lam} corresponding to $\lambda = \lambda_k$; note again that $b_k^{\pm} \to \infty$ as $k \to \infty$. We again use $C$ for a sufficiently large constant which may vary from line to line, and $C_\eps$ when this constant may depend on $\eps$. We will show that
\begin{equation}\label{eq.t3.2.cosh.goal}
	\int_0^\infty \cosh(2p(x))u_k(x)^2\,dx = (1+o(1) + \BigO(\sqrt{\eps}))\int_0^{x_{\lambda_k}} \frac{\pi \cosh(2p(x))\,dx}{\sqrt{\lambda_k - q(x)}}, \quad k \to \infty,
\end{equation}
then that
\begin{equation}\label{eq.t3.2.exp.goal}
	\int_0^{x_{\lambda_k}} \frac{\cosh(2p(x))\,dx}{\sqrt{\lambda_k-q(x)}} = (1+o(1))\int_0^{x_{\lambda_k}} \frac{e^{2p(x)}\,dx}{2\sqrt{\lambda_k - q(x)}}, \quad k \to \infty.
\end{equation}
This suffices to establish the theorem since \eqref{eq.t3.2.copy}, \eqref{eq.t3.2.cosh.goal}, and \eqref{eq.t3.2.exp.goal} combined give
\begin{equation}
	\mathop{\lim_{\eps \to 0}}_{\eps > 0} \limsup_{k \to \infty} \left|\|P_k\|\frac{\int_0^{x_{\lambda_k}}(\lambda_k-q(x))^{-1/2}\,dx}{\int_0^{x_{\lambda_k}} e^{2p(x)}(\lambda_k-q(x))^{-1/2}\,dx} - \frac{1}{2}\right| = 0.
\end{equation}

Our goal is to show that errors may be dominated by the integral on the right-hand side of \eqref{eq.t3.2.cosh.goal} over the interval from $x_{\lambda_k/3} = q^{-1}(\lambda_k/3)$ to $b_k^-$. We therefore follow the strategy of \eqref{eq.qprime.max} and obtain
\begin{equation}
	\begin{aligned}\label{eq.D.qprime}
	D := \int_{x_{\lambda_k/3}}^{b_k^-} \frac{\pi \cosh(2p(x))\,dx}{\sqrt{\lambda_k -q(x)}} & = \int_{x_{\lambda_k/3}}^{b_k^-} \frac{q'(x)\pi \cosh(2p(x))\,dx}{q'(x)\sqrt{\lambda_k-q(x)}}
	\\ & \geq \frac{\pi \cosh(2p(x_{\lambda_k/3}))}{q'(b_k^-)}\left(\sqrt{\frac{2}{3}} - \sqrt{\eps}\right)\sqrt{\lambda_k}
	\\ & \geq \frac{1}{C}\frac{\sqrt{\lambda_k}\cosh(2p(x_{\lambda_k/3}))}{q'(b_k^-)}.
	\end{aligned}
\end{equation}
Using that $q'(x)/q(x) = \BigO(1/x)$, we also have that
\begin{equation}\label{eq.D.noqprime}
	D \geq \frac{1}{C}\frac{q(b_k^-)\cosh(2p(x_{\lambda_k/3}))}{(1-\eps)\sqrt{\lambda_k} q'(b_k^-)} \geq \frac{1}{C} \lambda_k^{-1/2}b_k^- \cosh(2p(x_{\lambda_k/3})).
\end{equation}
We also use that, following \eqref{bp.1} and \eqref{bp.2}, for $k$ sufficiently large,
\begin{equation}
	\frac{b_k^+ - x_{\lambda_k/3}}{x_{\lambda_k/3}} \leq \frac{3(1+\eps - \frac{1}{3})q(x_{\lambda_k/3})}{x_{\lambda_k/3}q'(x_{\lambda_k/3})} \leq C,
\end{equation}
so our hypothesis \eqref{p'.slow} implies that
\begin{equation}\label{eq.slow.p.exp.ratio}
	\frac{\cosh(2p(b_k^+))}{\cosh(2p(x_{\lambda_k/3}))} \leq 2e^{2p(b_k^+)-2p(x_{\lambda_k/3})}(1+o(1)) \leq Ce^{C\frac{b_k^+ - x_{\lambda_k/3}}{x_{\lambda_k/3}}} \leq C.
\end{equation}

From \eqref{int.est.inf} and \eqref{eq.D.noqprime}, we have 
\begin{equation}\label{eq.t3.2.cosh.1}
	\int_{b_k^+}^\infty \cosh(2p(x))u_k(x)^2\,dx \leq C_\eps e^{- \frac{1}{C_\eps}\sqrt{\lambda_k}x_{\lambda_k}} = o(D), \quad k \to \infty.
\end{equation}
For the interval $[b_k^-, b_k^+)$, we use Corollary \ref{cor:u.bpm} and follow \eqref{eq.qprime.max}. Taking into account \eqref{eq.D.qprime} and \eqref{eq.slow.p.exp.ratio}, we see that
\begin{multline}\label{eq.t3.2.cosh.2}
\int_{b_k^-}^{b_k^+} \cosh(2p(x)) u_k(x)^2\,dx
\leq 
C \cosh(2p(b_k^+)) \int_{b_k^-}^{x_\lambda} \frac{dx}{\sqrt{\lambda_k - q(x)}} 
\\
\leq
\frac{C \sqrt{\varepsilon \lambda_k}}{q'(b_k^-)}  \cosh(2p(b_k^+)) \leq C\sqrt{\eps}D.
\end{multline}
To replace $u(x)^2$ with $(\lambda_k-q(x))^{-1/2}$ on the interval $[0, b_k^-)$, we apply Proposition \ref{prop:u.inside} with $v(x) = \cosh(2p(x))$; note that $v'$ is then positive on $[0, b_k^-)$ and so
\begin{equation}
	\|v'\|_{L^1((0, b_k^-))} = v(b_k^-) - v(0) = \cosh(2p(b_k^-)) - 1.
\end{equation}
Using that $p(b_k^-)\to\infty$ as $k\to\infty$, as well as \eqref{eq.D.noqprime} and \eqref{eq.slow.p.exp.ratio} with $p(b_k^-) \leq p(b_k^+)$, we see that
\begin{multline}\label{eq.t3.2.cosh.3}
	\left|\int_0^{b_k^-} \cosh(2p(x))u(x)^2\,dx - \int_0^{b_k^-} \frac{\pi \cosh(2p(x))\,dx}{\sqrt{\lambda_k - q(x)}}\right| \\ \leq \frac{C(1+|\log \eps|)}{\eps \lambda_k}\cosh(2p(b_k^-)) \leq C\frac{(1+|\log\eps|)}{\eps b_k^- \sqrt{\lambda_k}}D = o(D), \quad k\to\infty.
\end{multline}
Taking \eqref{eq.t3.2.cosh.1}, \eqref{eq.t3.2.cosh.2}, and \eqref{eq.t3.2.cosh.3} together proves \eqref{eq.t3.2.cosh.goal}.

The final step is to establish \eqref{eq.t3.2.exp.goal}. We have
\begin{equation}\label{cosh.approx.e2p.1}
	\begin{aligned}
	\int_0^{x_{\lambda_k}} \frac{e^{-2p(x)}}{\sqrt{\lambda_k-q(x)}}\,dx &\leq \sqrt{\frac{2}{\lambda_k}}\int_0^{x_{\lambda_k/2}} e^{-2p(x)}\,dx + e^{-2p(x_{\lambda_k/2})}\int_{x_{\lambda_k/2}}^{x_{\lambda_k}} \frac{dx}{\sqrt{\lambda_k - q(x)}}
		\\ &\leq \sqrt{\frac{2}{\lambda_k}}x_{\lambda_k/2} + e^{-2p(x_{\lambda_k/2})}\int_{x_{\lambda_k}/2}^{x_{\lambda_k}} \frac{dx}{\sqrt{\lambda_k - q(x)}}.
	\end{aligned}
\end{equation}
On the other hand, 
\begin{equation}\label{cosh.approx.e2p.2}
	\int_0^{x_{\lambda_k}} \frac{e^{2p(x)}}{\sqrt{\lambda_k-q(x)}}\,dx \geq e^{2p(x_{\lambda_k/2})}\int_{x_{\lambda_k/2}}^{x_{\lambda_k}} \frac{dx}{\sqrt{\lambda_k - q(x)}}.
\end{equation}
This clearly dominates the second term of \eqref{cosh.approx.e2p.1}. Following \eqref{eq.qprime.max} and \eqref{eq.qprime.max.2},
\begin{equation}\label{cosh.approx.e2p.3}
	\begin{aligned}
	\sqrt{\frac{2}{\lambda_k}}x_{\lambda_k/2}\left(e^{2p(x_{\lambda_k/2})}\int_{x_{\lambda_k/2}}^{x_{\lambda_k}} \frac{dx}{\sqrt{\lambda_k - q(x)}}\right)^{-1} &\leq \sqrt{\frac{2}{\lambda_k}}x_{\lambda_k/2}\left(e^{2p(x_{\lambda_k/2})} \frac{\sqrt{2 \lambda_k}}{q'(x_{\lambda_k})}\right)^{-1}
	\\ &= \frac{q'(x_{\lambda_k})x_{\lambda_k/2}}{e^{2p(x_{\lambda_k/2})} \lambda_k}
	\\ & \leq \frac{Cx_{\lambda_k/2}}{x_{\lambda_k}e^{2p(x_{\lambda_k/2})}} \to 0, \quad k \to \infty.
	\end{aligned}
\end{equation}
The estimates in \eqref{cosh.approx.e2p.2} and \eqref{cosh.approx.e2p.3}, taken with \eqref{cosh.approx.e2p.1}, allow us to establish that
\begin{equation}\label{eq.t3.2.e2px.bigger}
	\int_0^{x_{\lambda_k}} \frac{e^{-2p(x)}\,dx}{\sqrt{\lambda_k - q(x)}} = o\left(\int_0^{x_{\lambda_k}} \frac{e^{2p(x)}\,dx}{\sqrt{\lambda_k - q(x)}}\right).
\end{equation}
As a consequence, \eqref{eq.t3.2.exp.goal} holds, and the proof of the theorem is complete.
\end{proof}

\begin{remark}\label{rem.t3.2.asymp.inv}
We can modify \eqref{cosh.approx.e2p.1}--\eqref{cosh.approx.e2p.3} to show that the quotient in \eqref{e3Conclusion.2} only depends on the limiting behavior of $p(x)$ in the following sense.

Let $\lambda_k$ and $x_{\lambda_k}$ be the eigenvalues and turning points of $T$ as in Theorem \ref{t3.2}, and let both $p_1$ and $p_2$ satisfy Assumption \ref{asm:p} and
\begin{equation}\label{eq.rem.plimit.diff.to.zero}
	p_1(x) - p_2(x) = o(1), \quad x\to\infty.
\end{equation}
Then
\begin{equation}\label{eq.rem.plimit.statement}
	\lim_{k\to\infty} \frac{\int_0^{x_{\lambda_k}} e^{2p_1(x)}(\lambda_k - q(x))^{-1/2}\,dx}{\int_0^{x_{\lambda_k}} e^{2p_2(x)}(\lambda_k - q(x))^{-1/2}\,dx} = 1.
\end{equation}

For a sequence $\{c_k\}_{k\in\Bbb{N}}$ with $0 \leq c_k \leq x_{\lambda_k/2}$ to be determined, write
\begin{multline}\label{eq.rem.plimit.spread}
	\int_0^{x_{\lambda_k}} \frac{e^{2p_1(x)}}{\sqrt{\lambda_k - q(x)}}\,dx = \int_0^{c_k} \frac{e^{2p_1(x)}}{\sqrt{\lambda_k-q(x)}}\,dx + \int_{c_k}^{x_{\lambda_k}} \frac{e^{2p_1(x)}}{\sqrt{\lambda_k-q(x)}}\,dx
	\\ \leq e^{2p_1(c_k)}\int_0^{c_k} \frac{dx}{\sqrt{\lambda_k-q(x)}} + e^{2\sup_{x \geq c_k}(p_1(x)-p_2(x))}\int_{c_k}^{x_{\lambda_k}} \frac{e^{2p_2(x)}}{\sqrt{\lambda_k-q(x)}}\,dx.
\end{multline}
Since $(\lambda_k - q(x))^{-1/2} \leq (\lambda_k/2)^{-1/2}$ for $x \in [0, c_k] \subset [0, x_{\lambda_k/2}],$
\begin{equation}
	\frac{\int_0^{c_k}(\lambda_k - q(x))^{-1/2}\,dx}{\int_0^{x_{\lambda_k}} (\lambda_k - q(x))^{-1/2}\,dx} \leq \frac{c_k(\lambda_k/2)^{-1/2}}{x_{\lambda_k}\lambda_k^{-1/2}} = \frac{\sqrt{2}c_k}{x_{\lambda_k}}.
\end{equation}
Inserting this into \eqref{eq.rem.plimit.spread} and using $p_2(x) \geq 0$, so
\begin{equation}
	\int_0^{x_{\lambda_k}} (\lambda_k - q(x))^{-1/2}\,dx \leq \int_0^{x_{\lambda_k}} e^{2p_2(x)}(\lambda_k - q(x))^{-1/2}\,dx,
\end{equation}
we see that
\begin{multline}\label{eq.rem.plimit.tobound}
	\int_0^{x_{\lambda_k}} \frac{e^{2p_1(x)}}{\sqrt{\lambda_k - q(x)}}\,dx 
	\\ \leq \left(\frac{\sqrt{2}c_k e^{2p_1(c_k)}}{x_{\lambda_k}} + e^{2\sup_{x \geq c_k}(p_1(x)-p_2(x))}\right) \int_0^{x_{\lambda_k}} \frac{e^{2p_2(x)}}{\sqrt{\lambda_k-q(x)}}\,dx.
\end{multline}

Because $x_{\lambda_k/2}, x_{\lambda_k} \to \infty$ as $k \to \infty$ we may choose $c_k \to \infty$ sufficiently small so that $\frac{\sqrt{2}c_k e^{2p_1(c_k)}}{x_{\lambda_k}} \to 0$, and since $p_1(x)-p_2(x) = o(1)$, we have for this choice of $c_k$ that 
\begin{equation}
e^{2\sup_{x \geq c_k}(p_1(x)-p_2(x))} + \frac{\sqrt{2}c_k e^{2p_1(c_k)}}{x_{\lambda_k}}
 = 1+o(1), \quad k \to \infty.
\end{equation}
Putting this in \eqref{eq.rem.plimit.tobound} and reversing the roles of $p_1$ and $p_2$ gives \eqref{eq.rem.plimit.statement}.

Therefore, if $p_1$ and $p_2$ satisfy \eqref{eq.rem.plimit.diff.to.zero} and the hypotheses of Theorem \ref{t3.2}, then the asymptotics for the norms of the corresponding spectral projections are the same, that is, their ratio tends to 1.
\end{remark}

\begin{corollary}\label{cor:Pk.beta.slow}
Let $q(x)= |x|^\beta$, $\beta \geq 2$, let $p$ satisfy Assumption \ref{asm:p}, let $L$ be defined as \eqref{e3LDef}, let $\{P_k\}_k$ be its spectral projections, let $\Omega_\beta$ be as in \eqref{e3klambdak} and let 
\begin{equation}
	B(x, y) = \int_0^1 t^{x-1}(1-t)^{y-1}
\end{equation}
be the beta function.
\begin{enumerate}[{\upshape i)}]
\item\label{cor.slow.i} For $\gamma > 0$, let
\begin{equation}
p(x) = \frac{1}{2}\gamma \log x, \quad \forall x > 2.
\end{equation}		
Then
\begin{equation}\label{eq.cor3.2.i}
\lim_{k \rightarrow \infty} \frac{\|P_k\|}{k^{\frac{2}{2+\beta}\gamma}} 
= 
\frac12 
\left(
\frac{\pi}{2 \Omega_\beta}
\right)^{\frac{2}{2 + \beta}\gamma} \frac{B(\frac{\gamma+1}{\beta}, \frac12)}{B(\frac{1}{\beta}, \frac12)} .
\end{equation}
\item For $\gamma > 0$, let
\begin{equation}
p(x) = \frac{1}{2}\gamma \log \log x \quad \forall x > e^2.
\end{equation}
Then
\begin{equation}
\lim_{k \rightarrow \infty} \frac{\|P_k\|}{(\log k)^{\gamma}} 
= 
\frac12 \left(\frac{2}{2+\beta}\right)^{\gamma}.
\end{equation}
\end{enumerate}
\end{corollary}
\begin{proof}
	Following Remark \ref{rem.t3.2.asymp.inv}, modifying $p(x)$ on a bounded set does not change the asymptotics in Theorem \ref{t3.2}. For $p(x) = \frac{1}{2}\gamma\log x$, we change variables twice to compute
	\begin{equation}
	\begin{aligned}
	\int_0^{x_{\lambda_k}} \frac{e^{2p(x)}\,dx}{\sqrt{\lambda_k - q(x)}} &= \int_0^{x_{\lambda_k}} x^\gamma (x_{\lambda_k}^\beta - x^\beta)^{-1/2}\,dx
	\\ &= x_{\lambda_k}^{1+\gamma-\frac{\beta}{2}}\int_0^1 x^\gamma (1-x^\beta)^{-1/2}\,dx
	\\ &= \beta x_{\lambda_k}^{1+\gamma-\frac{\beta}{2}} B\left(\frac{1+\gamma}{\beta}, \frac{1}{2}\right).
	\end{aligned}
	\end{equation}
	Combining with the case $\gamma = 0$, along with Theorem \ref{t3.2}, gives us that
	\begin{equation}
		\lim_{k\to\infty} \|P_k\|\frac{B(\frac{1}{\beta}, \frac{1}{2})}{B(\frac{1+\gamma}{\beta}, \frac{1}{2})x_{\lambda_k}^\gamma} = \frac{1}{2}.
	\end{equation}
	The result \eqref{eq.cor3.2.i} then comes from the asymptotics \eqref{xlambdak.qbeta.asym} for $x_{\lambda_k}$.

	For the second computation, when $p(x) = \frac{1}{2}\gamma \log \log x$, we compute
	\begin{equation}
	\begin{aligned}
	\int_0^{x_{\lambda_k}} \frac{e^{2p(x)}\,dx}{\sqrt{\lambda_k - q(x)}} &= \int_0^{x_{\lambda_k}} (\log x)^\gamma (x_{\lambda_k}^\beta - x^\beta)^{-1/2}\,dx
	\\ &= x_{\lambda_k}^{1-\frac{\beta}{2}} \int_0^1 (\log x_{\lambda_k} + \log x)^\gamma (1-x^\beta)^{-1/2}\,dx
	\\ &= (\log x_{\lambda_k})^\gamma (1+o(1))\int \frac{dx}{\sqrt{\lambda_k - q(x)}}, \quad k \to \infty.
	\end{aligned}
	\end{equation}
	The result follows from Theorem \ref{t3.2} and \eqref{xlambdak.qbeta.asym}.
\end{proof}

\section{Additional remarks}\label{s4Remarks}

\subsection{When $p$ is complex-valued, non-odd, or non-smooth}\label{secNonOdd}

We begin by noting that the analysis in Section \ref{s3Monomial} is completely insensitive to an imaginary part of $p$.  
Recalling the conjugation \eqref{e3Conjugation}, we can write
\begin{equation}\label{e4Complex1}
 	T+B(p+ir) = W(T+B(p))W^*
\end{equation}
for the unitary operator on $L^2(\R)$ 
\begin{equation}\label{e4Complex2}
 	Wu(x) = e^{ir(x)}u(x).
\end{equation}
Therefore so long as $r \in C^{2}(\R)$ is a real-valued function, we can define the operator $T+B(p+ir)$ via \eqref{e4Complex1} and \eqref{e4Complex2} and note that its spectral properties are determined entirely by the real part $p$. 

We also note that the assumption that $p$ is odd in most of Sections \ref{sec.ho} and \ref{s3Monomial} was not essential.  Without it, $\|f_k\|\neq\|g_k\|$, but one can still find $\|f_k\|$ and $\|g_k\|$ separately and compute the spectral projection norms via \eqref{Pk.fk.gk}.

To illustrate the different asymptotic regimes which appear, we present an example where $q(x)=x^2$ and $p(x)$ is even and $p(x)\rightarrow +\infty$ as $|x|\rightarrow \infty$.  The principal contribution to $\|g_k\|$ therefore comes from $x$ small compared with $x_{\lambda_k} = \sqrt{2k+1}$.

Therefore, let $T$ be the harmonic oscillator as in Section \ref{sec.ho}, let
\begin{equation}\label{e4pDefEven}
 	p(x) = a\sqrt{1+x^2}, \quad a > 0.
\end{equation}
and let the perturbation $B(p)$ be defined by formula \eqref{e3BDef}.
We remark that equally sharp asymptotics could be found assuming only that $p\in C^2(\R)$, that $p'$ and $p''$ are bounded, and that $p(x) = a_\pm x + \BigO(|x|^{-1})$ as $x \rightarrow \pm \infty$ with $a_\pm \neq 0$.

\begin{proposition}\label{p4Even}
 	Let $P_k$ denote the spectral projection for $L$, defined in \eqref{e3LDef} with $p$ as in \eqref{e4pDefEven}, for the eigenvalue $\lambda_k = 2k+1$.  Then we have the asymptotic formula
	\begin{equation}\label{e4pEven}
		\|P_k\| = \frac{\|e^{-p(x)}\|}{\pi^{3/4}(2k)^{3/8}a^{1/4}} \exp\left(a\sqrt{2k}\right)(1+\BigO(k^{-1/2})).
	\end{equation}
\end{proposition}

\begin{proof} 
	In order to analyze $\|f_k\|=\|e^{p} h_k\| $, begin by introducing
	\begin{equation}\label{e4PsiDef}
		\Psi_a := \int_{-\infty}^\infty e^{2a|x|}|h_k(x)|^2\,dx.
	\end{equation}
	Using that $|h_k(x)|^2$ is even,  we have that
	\[
		\begin{aligned}
		\Psi_a &= 2\int_0^\infty e^{2ax}|h_k(x)|^2\,dx
		& = 2\int_{-\infty}^\infty e^{2ax}|h_k(x)|^2\,dx - 2\int_{-\infty}^0 e^{2ax}|h_k(x)|^2\,dx.
		\end{aligned}
	\]
	Recall that $a > 0$ and that the $h_k$ are normalized in $L^2(\R)$, and therefore the third integral is bounded by 1. Asymptotics for the second integral are computed in Theorem \ref{t2} where the integral appears as $\|P_k\|$, and so we conclude that
	\begin{equation}\label{e4PsiValue}
		\Psi_a = \frac{1}{(2k)^{1/4}\sqrt{a\pi}}\exp\left(2^{3/2}a\sqrt{k}\right)\left(1+\BigO(k^{-1/2})\right).
	\end{equation}
	Then
	\[
		\|f_k\|^2 = \Psi_a + 2\int_0^\infty e^{2a|x|}\left(e^{2(p(x)-a|x|)}-1\right)|h_k(x)|^2 \,dx.
	\]
	It is easy to use the Taylor series to show that
	\[
		e^{2(p(x)-a|x|)}-1 = \frac{a}{x} + \BigO(x^{-2}), \quad x\rightarrow +\infty.
	\]
	
	For an upper bound for the error $\|f_k\|^2 - \Psi_a$, we only need to note that the contribution when $|x| \leq \sqrt{k}$ is negligible compared with $\Psi_a$ since, using that $2^{3/2} - 2 > 0.8$,
	\[
		2\int_0^{\sqrt{k}} e^{2a|x|}\left(e^{2(p(x)-a|x|)}-1\right)|h_k(x)|^2 \,dx  \leq 2e^{2a\sqrt{1+k}}\|h_k\|^2 \ll e^{-0.8a\sqrt{k}}\Psi_a.
	\]
	The same reasoning shows that $\Psi_a$ as defined in \eqref{e4PsiDef} is determined by the integral over $\{|x| \geq \sqrt{k}\}$ with exponentially small relative error. Therefore 
	\begin{align*}
		2\int_0^\infty  e^{2a|x|} & \left(e^{2(p(x)-a|x|)}-1\right)|h_k(x)|^2 \,dx 
		\\ &=  2\int_{\sqrt{k}}^{\infty} e^{2a|x|}\left(e^{2(p(x)-a|x|)}-1\right)|h_k(x)|^2 \,dx +\BigO(e^{-0.8a\sqrt{k}}\Psi_a)
		\\ &\leq 2\int_{\sqrt{k}}^{\infty} e^{2a|x|}\left(\frac{a}{\sqrt{k}}+\BigO(k^{-1})\right)|h_k(x)|^2\,dx + \BigO(e^{-0.8a\sqrt{k}}\Psi_a)
		\\ &\leq \BigO(k^{-1/2}\Psi_a).
	\end{align*}
	We conclude that
	\begin{equation}\label{e4fkEven}
		\|f_k\|^2 = \Psi_a(1+\BigO(k^{-1/2})).
	\end{equation}
	We turn to analysis of $g_k(x)=e^{-p(x)}h_k(x)$. From Propositions \ref{prop:u.exp}, \ref{prop:u.inside}, Corollary \ref{cor:u.bpm} and the fact that $e^{-2p(x)}$ is exponentially small on $\{|x| \geq x_{\lambda_k}^{1/2} \}$ it is clear that
	\[
		\|g_k\|^2 = \int_{0}^{b_k^ -} \frac{e^{-2p(x) } \, dx}{\pi \sqrt{\lambda_k-x^2}} + \BigO(k^{-1})
		= \int_{0}^{x_{\lambda_k}^{1/2}} \frac{e^{-2p(x) }\, dx}{\pi \sqrt{\lambda_k-x^2}}+ \BigO(k^{-1});
	\]
	notice that, since the $h_k$ are normalized, the equivalent of Proposition \ref{prop:u.inside} reads
	\begin{equation}
	\begin{aligned}
	& \left|\int_0^{b_k^-} v(x) h_k(x)^2\,dx - \int_0^{b_k^-} \frac{v(x)\, dx}{\pi\sqrt{\lambda_k - x^2}}\right|
	\\
	& \qquad \qquad \leq 
	\frac {1}{\varepsilon \lambda_k}
	\left(\left(C_2 + \frac{3}{4}|\log \eps|\right)
	\|v\|_{L^{\infty}((0,b_k^-))}
	+ 
	\|v'\|_{L^1((0,b_k^-))}
	\right)
	\end{aligned}
	\end{equation}
	On the interval $[0, x_{\lambda_k}^{1/2}]$, we have
	\[
		\frac{1}{\sqrt{\lambda_k-x^2}} = \frac{1}{x_{\lambda_k}}(1- (x/x_{\lambda_k})^2)^{-1/2} = \frac{1}{\sqrt{2k}}(1+\BigO(k^{-1/2})).
	\]
	Therefore, using the rapid decay of $e^{-2p(x)}$, 
	\begin{equation}
	\begin{aligned}
		\|g_k\|^2 &=
		\int_{0}^{x_{\lambda_k}^{1/2}} \frac{e^{-2p(x)} \, dx}{\pi \sqrt{\lambda_k-x^2}} 
		= 
		\left(\frac{1}{ \pi \sqrt{2k}}\int_{0}^{x_{\lambda_k}^{1/2}} e^{-2p(x)}\,dx \right)(1+\BigO(k^{-1/2}))
		\\
		&= \frac{1}{ \pi \sqrt{2k}} \|e^{-p}\|^2(1+\BigO(k^{-1/2})).
	\end{aligned}
	\end{equation}
	Taking a square root and multiplying by \eqref{e4fkEven}, along with the definition \eqref{e4PsiDef} and value \eqref{e4PsiValue} of $\Psi_a$, proves the proposition.
\end{proof}

It is natural to study an even perturbation for which $\|e^{-p}\|$ is easy to compute, such as
\begin{equation}\label{e4pAbs}
	p(x) = a|x|, \quad a > 0,
\end{equation}
but complications arise upon calculating that the operator $B$ as in \eqref{e3BDef} would involve a delta function coming from the second derivative of $p$.

It is convenient to define $L$ via forms instead. We start by defining a form $l$ as a sum of the forms $t$ and $b$ associated with $T$ and a perturbation. Specifically,
\[
\begin{aligned}
l&= t +b, \\ 
t(f,g) &= \langle f', g' \rangle + \langle xf, xg \rangle, \\
b(f,g) & = a^2 \langle f,g \rangle + i (\langle Df,p'g \rangle + \langle p'f, Dg \rangle ),\\
\Dom(l) & = \Dom(b) = \Dom(t) = \{ f \in W^{1,2}(\R) : xf \in L^2(\R)  \}.
\end{aligned}
\]
Notice that integration by parts gives 
\[
b(f,g) = 2a f(0)\overline{g(0)} - a^2 \langle f,g \rangle +2 i \langle p'Df,g \rangle,
\]
where the first term corresponds to $2a \delta(x)$. The form $t$ is symmetric, bounded from below, and closed. Since 
\[
|b(f,f)| \leq a (2 \|f'\| + a \|f\|) \|f\| \leq C (t(f,f)+\|f\|^2)^{1/2} \|f\|,
\]
we see that $b$ is $s$-subordinated, in the sense of forms, to $t$ with $s=1/2$.  The form $b$ is therefore relatively bounded with respect to $t$ with bound $0$, just as in the operator case. It follows from \cite[Thm.\ VI.3.4]{Kato-1966} that $l$ is a sectorial closed form on $\Dom(l) =\Dom(t)$ and it uniquely determines, via the first representation theorem \cite[Thm.\ VI.2.1]{Kato-1966}, an $m$-sectorial operator $L$ with compact resolvent. The action and the domain of $L$ can be obtained explicitly following the lines of \cite[Ex.\ VI.2.16, Sec.\ VI.4.1]{Kato-1966}; they are
\begin{equation}\label{e4LAbsDef}
\begin{aligned}
L  & = - \frac{d^2}{d x^2} +x^2 - a^2  +2 i a \opnm{sgn} (x) D,  
\\
\Dom(L) & = \{ f \in W^{2,2}(\R \setminus \{0\} ): x^2 f \in L^2(\R), 
\\
& \quad \quad f(0+) = f(0-), 
f'(0+) - f'(0-) = 2a f(0)  \}.
\end{aligned}
\end{equation}
Notice that any $f \in \Dom(L)$ has a continuous extension in $W^{1,2}(\R)$. We can check that the functions $f_k(x) = e^{a|x|} h_k(x)$ belong to $\Dom(L)$, and they are in fact eigenfunctions corresponding to eigenvalues $2k+1$ as expected. Moreover, the arguments used to prove Lemma 2.5 show that the set $\{f_k\}_{k=0}^\infty$ is complete in $L^2(\R)$. 

The adjoint operator $L^*$ is associated with the adjoint form $l^ *(f,g):=\overline{l(g,f)}$, and it can be verified using the above arguments that
\[
\begin{aligned}
L^*  & = - \frac{d^2}{d x^2} +x^2 - a^2  - 2 i a \opnm{sgn}(x) D  
\\
\Dom(L^*) & = \{ f \in W^{2,2}(\R \setminus \{0\} ): x^2 f \in L^2(\R), 
\\
& \quad \quad f(0+) = f(0-), 
f'(0+) - f'(0-) = -2a f(0)  \}.
\end{aligned}
\]
Then the functions $g_k=e^{-a|x|}h_k(x)$ are eigenfunctions of $L^*$, and they also form a complete set.

The same computations which proved Proposition \ref{p4Even} may then be applied to $L$ as in \eqref{e4LAbsDef} with $a > 0$, but here we can easily compute $\|e^{-p(x)}\| = 1/a$.

\begin{proposition}\label{p4Abs}
 	Let $P_k$ denote the spectral projection for $L$, defined in \eqref{e4pAbs} and \eqref{e4LAbsDef}, for the eigenvalue $\lambda_k = 2k + 1$.  Then we have the asymptotic formula
	\begin{equation}
		\|P_k\| = \frac{1}{(\pi a)^{3/4} (2k)^{3/8}} \exp\left(a\sqrt{2k}\right)(1+\BigO(k^{-1/2})).
	\end{equation}
\end{proposition}

\subsection{Existence of an Abel basis with brackets}
Finally, we also compare our operators with $q(x)=|x|^ \beta$, $\beta \geq 2$, and $p$ obeying \eqref{e3pBound}  in Section \ref{s3Monomial} with the main theorem in \cite{Agr1994}, which provides sufficient conditions for an operator to admit various types of bases with brackets comprised of root vectors. For definitions, see \cite[Sec.\ I.1.6]{Markus-1988}.
Adding a constant $C_0 > 0$ sufficiently large, we may see that the imaginary part of $L+C_0$ is subordinated to the real part in the sense of quadratic forms, see \cite[Eq.\ (1.6)]{Agr1994}  with $s$ in \eqref{e3sDef}. We then may verify that
\[
\limsup_{j\rightarrow\infty} s_j(\Re(L+C_0)) j^{-h} > 0
\]
if and only if $h \leq 2\beta/(2+\beta)$; cf.\ \cite[Eq.\ (1.7)]{Agr1994}.

We find ourselves in the situation of the main theorem \cite[Eq.\ (1.8.c)]{Agr1994} with
\[
p(1-q) = \frac{2\beta}{2+\beta}\left(\frac{1}{2}-\max\left\{\frac{\alpha-1}{\beta},0\right\}\right) \leq \frac{\beta}{2+\beta} < 1.
\]
We may rephrase this in the following proposition.

\begin{proposition}
	Consider the operator $L$ in \eqref{e3LDef} with $q(x)=|x|^ \beta$, $\beta \geq 2$, and $p$ obeying \eqref{e3pBound}.  The collection of root vectors $\{f_k\}_{k=0}^\infty$ defined in \eqref{e3fkDef} is an Abel basis with brackets of order $\gamma$ for any
	\[
	\gamma > \frac{1}{\beta} + \max \left\{\frac{\alpha-1}{\beta},0\right\},
	\]
	and the same holds for $\{g_k\}_{k=0}^\infty$.
\end{proposition}

Recall that the existence of a Riesz basis, instead of an Abel basis, is impossible when the spectral projections norms are unbounded. That the root vectors of an operator may fail to be a Riesz basis, or even an Abel basis, is typically connected with resolvent growth away from the eigenvalues of the operator; see \cite{TreEmb2005} and the many references therein.

\subsection{Existence of an embedded Riesz basis}
\label{subsec:RB}

Several recent papers study basis property and eigenvalue asymptotics of perturbations of operators resembling the harmonic and anharmonic oscillators. Namely, the unperturbed operator $T$ is self-adjoint and non-negative, the spectrum of $T$ is discrete, all eigenvalues $\{\mu_k\}_{k \in \N}$ are simple and 
\begin{equation}\label{ev.gaps}
\lambda_{k+1} - \lambda_k \geq \kappa k^{\alpha-1},
\end{equation}
with $\kappa>0$ and $\alpha>1/2$. 
If the perturbation $B$ satisfies the so-called local subordination condition

\begin{subequations}\label{loc.so}
	\begin{equation}\label{loc.so.1}
\lim_{k \to \infty}\|B \psi_k\|  \leq c_k k^{\alpha -1},
\end{equation}
\begin{equation}\label{loc.so.2}
\lim_{k \to \infty} c_k =0,
\end{equation}
\end{subequations}
where $\{\psi_k\}_{k \in \N}$, $\|\psi_k\|=1$, are eigenfunctions of $T$ related to eigenvalues $\lambda_k$,
then the eigensystem of $T+B$ contains a Riesz basis.

The conditions \eqref{loc.so} with $\alpha=1$ appeared first in \cite{AddMit2012}, where perturbations of the harmonic oscillator by multiplication by a complex-valued function were studied and it was proved that the eigensystem of $T+B$ contains a Riesz basis. In \cite{Shkalikov-2010-269}, it was proved in an abstract setting that if $\alpha =1$ and $B$ satisfies \eqref{loc.so.1} and, instead of \eqref{loc.so.2}, $\{c_k\}_{k\in\N}$ is assumed to a bounded sequence, then the eigensystem of $T+B$ contains a basis with brackets. Moreover, the asymptotics of the eigenvalue-counting function of $T+B$ was proved in \cite{Shkalikov-2010-269} as well; see also \cite{Shkalikov-2012-18} for further generalizations. The next step, in particular showing the existence of a Riesz basis for $\alpha >1/2$ and $B$ satisfying \eqref{loc.so}, was done in \cite{AddMit2012a}. Finally, a sesquilinear form version of \eqref{loc.so} was used in \cite{Mityagin-2015} to deal with singular perturbations of the harmonic oscillator; eigenvalue asymptotics for the latter were further studied in \cite{Mityagin-2014a,Mityagin-2015-54}. 

For the harmonic oscillator and perturbations by potentials, it was showed in \cite{AddMit2012,AddMit2012a,Mityagin-2015} that the eigensystem of $T+V$ contains a Riesz basis if $V \in L^p(\R)$ with $1 \leq p < \infty$ or $\|V\|_\infty<1$. 
It is not known whether the condition $V \in L^\infty(\R)$, without the assumption $\|V\|_\infty < 1$, is sufficient to conclude that the eigensystem of $T+V$ contains a Riesz basis.

Notice that our Theorem \ref{t3} shows that the condition \eqref{loc.so} does not hold for $T$ as in \eqref{e3TDef} and the perturbation $B(p)$ as in \eqref{e3BDef} with $q$ and $p$ satisfying Assumptions \ref{asm:q} and \ref{asm:p}, respectively, since the norms of the spectral projections are unbounded. Nonetheless, our Theorem \ref{t3.2} and results summarized above enable us to analyse $T$ with $q(x)=|x|^\beta$, $\beta > 2$, perturbed by certain unbounded complex potentials.
\begin{proposition}
Let $T$ be as in \eqref{e3TDef} with $q(x)= |x|^\beta$, $\beta > 2$. If, for some $\eps, C_0 >0$,  
\begin{equation}\label{V.RB}
|V(x)| \leq C_0 \left(1 + |x|^{\frac{\beta-2}{2}-\eps}\right),
\end{equation}
then the eigensystem of  $T+V$ contains a Riesz basis.
\end{proposition}
\begin{proof}
Without loss of generality, we may assume that $\beta - 2 - 2\eps > 0$, since when \eqref{V.RB} holds, the same statement holds for any smaller $\tilde{\eps}\in (0, \eps)$. Let $\{y_k\}_{k=0}^\infty$ be the eigenfunctions of $T$ as in \eqref{yu.rel}. In view of \eqref{loc.so.1} and \eqref{loc.so.2}, our goal will be to show that
\begin{equation}\label{TplusV.goal}
\lim_{k\to\infty}\frac{\|Vy_k\|}{\|y_k\|}k^{1-\alpha} = 0,
\end{equation}
with $\alpha-1 = \frac{\beta-2}{\beta+2}$ the spacing between eigenvalues of $T$ in \eqref{ev.gaps}; see e.g.\ \cite[Eq.~(8.10)]{AddMit2012a}. Letting $C > 0$ vary from instance to instance, we have that
\begin{equation}
\begin{aligned}
\|V y_k\|^2
&\leq 
C \int_0^\infty \left(1 + |x|^{\frac{\beta-2}{2}-\eps}\right)^2 y_k(x)^2 dx
\\
&\leq
C \|y_k\|^2
+
C \int_1^\infty \cosh\left(2\left(\frac{\beta-2}{2}-\eps\right) \log x\right) y_k(x)^2 dx.
\end{aligned}
\end{equation}	

As in Corollary \ref{cor:Pk.beta.slow}.\ref{cor.slow.i}, we may find some $p(x)$ satisfying Assumption~\ref{asm:p} such that
\begin{equation}
p(x) = \frac{1}{2}(\beta-2-2\eps) \log x, \quad x > C.
\end{equation}
This is so that, when $L$ is defined as \eqref{e3LDef} and $\{P_k\}_k$ denote the spectral projections of $L$, by \ref{cor:Pk.beta.slow}.\ref{cor.slow.i} and the reasoning leading to \eqref{eq.t3.2.copy},
\begin{equation}
	\|P_k\| = (1+o(1)) \frac{\|\cosh(2p(x))y_k(x)\|^2}{\|y_k\|^2} \leq Ck^{\frac{2}{2+\beta}(\beta-2-2\eps)}, \quad k \geq 1.
\end{equation}
Taking square roots and using that $\beta-2-2\eps > 0$,
\begin{equation}
	\frac{\|Vy_k\|}{\|y_k\|} \leq C + C(1+o(1))\frac{\|\cosh(2p)y_k\|}{\|y_k\|} \leq Ck^{\frac{\beta-2-2\eps}{2+\beta}}.
\end{equation}
The conclusion \eqref{TplusV.goal} follows since $k^{\frac{\beta-2-2\eps}{2+\beta} + 1-\alpha} = k^{-\frac{2\eps}{2+\beta}}$, so condition \eqref{loc.so} is satisfied. Hence the claim follows from \cite[Thm.~6]{AddMit2012a}.
\end{proof}

\appendix

\section{}

In order to discuss solutions to the second order ODE which may not be in $L^2(\R)$, we first include a theorem from \cite{Olv1997}. To state the theorem we use the variation of a function $f\in C^1(\Bbb{R})$ over the interval $(a,b)$:
\[
\mathcal{V}_{a,b}(f) = \int_a^b |f'(x)|\,dx.
\]

\begin{theorem}\emph{\cite[Chapter 6, Theorem 2.1]{Olv1997}}\label{tAOlver}
	In a given finite or infinite interval $(a_1, a_2)$, let $f(x)$ be a positive, real, twice continuously differentiable function, $g(x)$ a continuous real or complex function, and
	\[
	F(x) = \int \left(\frac{1}{f^{1/4}}\frac{d^2}{dx^2}\left(\frac{1}{f^{1/4}}\right)-\frac{g}{f^{1/2}}\right)\,dx.
	\]
	Then in this interval the differential equation
	\[
	d^2w/dx^2 =(f(x)+g(x))w
	\]
	has twice continuously differentiable solutions
	\[
	w_1(x) = f^{-1/4}(x) \exp\left(\int f^{1/2}(x)\,dx\right)(1+\eps_1(x)),
	\]
	\[
	w_2(x) = f^{-1/4}(x) \exp\left(-\int f^{1/2}(x)\,dx\right)(1+\eps_2(x)),
	\]
	such that
	\[
	|\eps_j(x)|, \frac{1}{2}f^{-1/2}|\eps_j'(x)| \leq \exp\left(\frac{1}{2}\mathcal{V}_{a_j,x}(F)\right)-1, \quad j=1,2,
	\]
	provided that $\mathcal{V}_{a_j,x}(F) < \infty$.  If $g(x)$ is real, then the solutions are real.
\end{theorem}

More precise estimates of eigenfunctions $\{y_k\}_k$ of the operator $T$ defined in \eqref{e3TDef} can be obtained by the method of Langer, see e.g.\ \cite[\S 22.27]{Titchmarsh-1958-book2}. 
We discuss asymptotics in two regimes: beyond the turning point $x_\lambda$ (defined by \eqref{xlambda.def}) and strictly between the turning points $-x_\lambda$ and $x_\lambda$.
Since $q$ is even, it suffices to work on $(0,+\infty)$. We follow the notations of \cite{Giertz-1964-14}:
\begin{equation}\label{u.def}
\begin{aligned}
q(x_\lambda) &= \lambda,\quad  x_\lambda > 0, \\
a_\lambda & =q'(x_\lambda),
\\
\zeta &= \zeta(x,\lambda)= 
\begin{cases}
\int_x^{x_\lambda} \sqrt{\lambda - q(s)} \, d s, & \mbox{ for } 0<x<x_\lambda, 
\\
i \int_{x_\lambda}^x \sqrt{q(s) - \lambda} \, d s, & \mbox{ for } x > x_\lambda,
\end{cases}
\\
u & = u(x,\lambda) = \left(\frac{\zeta}{\zeta'} \right)^{1/2} K_{1/3}(-i \zeta), \quad x>0, 
\\
u_k&= u(x, \lambda_k), \quad x>0, \quad  \lambda_k \in \opnm{Spec}(T),
\end{aligned}
\end{equation}
where $K_{1/3}$ is the modified Bessel function of order $1/3$. We also define the positive numbers $\delta$ and $\delta_1$ by
\begin{equation}\label{delta.def}
\zeta(x_\lambda-\delta) = - i \zeta(x_\lambda + \delta_1) = 1.
\end{equation}
\begin{lemma}[{\cite[\S 22.27]{Titchmarsh-1958-book2}}]
	\label{lem:titch}
	Let $q$ satisfy Assumption \ref{asm:q}, $x_\lambda$ be defined by \eqref{xlambda.def} and $u$ be as in \eqref{u.def}.
	Then there is a solution of $-y'' + (q-\lambda) y =0$ on $(0,+\infty)$ such that
	\begin{equation}\label{yu.rel.ap}
	y(x) = u(x) (1+ \BigO(x_\lambda^{-1} \lambda^{-1/2}))
	\end{equation}
	uniformly with respect to $x$. 
\end{lemma}

\begin{lemma}\label{lem:u.basic}
	Let $q$ satisfy Assumption \ref{asm:q}, $x_\lambda$ be defined by \eqref{xlambda.def} and $a_\lambda,\zeta,u,\delta, \delta_1$ be as in \eqref{u.def}--\eqref{delta.def}. Then
	\begin{enumerate}[{\upshape(i)}]
		\item $\delta$ and $\delta_1$ are both $\BigO(a_\lambda^{-1/3})$ as $\lambda \to + \infty$,
		\item \label{b.lam}  
		for $\varepsilon \in (0,1)$, define the positive numbers $b^\pm$ by equation
		\begin{equation}\label{bpm.def}
		q(b^\pm) = (1 \pm \varepsilon) \lambda;
		\end{equation}
		then there exists $\lambda_{\varepsilon}>0$ such that, for all $\lambda > \lambda_\varepsilon$, 
		\begin{equation}
		b^- \leq x_\lambda - \delta \quad \mbox{ and } \quad b^+ \geq x_\lambda + \delta_1,
		\end{equation}
		
		\item\label{u2.in} for $0 \leq x < x_{\lambda} - \delta$ and $\zeta >1$,
		\begin{equation}
		u^2 = \frac{\pi}{\sqrt{\lambda-q}} \left( 1 + \sin 2 \zeta + R_1(\zeta) \right),
		\end{equation}
		where $|R_1(\zeta)| < 1/(2\zeta)$,
		\item\label{u.exp.1} there exists $C_1>0$ such that, for all $\lambda \in q(\R)$,
		\begin{equation}\label{u.int.bound}
		|u(x)| \leq \frac{C_1}{(q(x)-\lambda)^{1/4}} \exp \left(-\int_{x_\lambda}^x \sqrt{q(s) - \lambda}  \, ds \right), \quad x> x_\lambda + \delta_1,
		\end{equation}
		\item as $\lambda \to \infty$, 
		\begin{equation}\label{u.norm}
		\begin{aligned}
		\int_0^{\infty} u(x)^2 \, dx 
		& =  \int_0^{x_\lambda} \frac{\pi \, d x}{\sqrt{\lambda - q(x)}} 
		+ \BigO(x_\lambda^{2/3} \lambda^{-2/3})
		\\
		& = 
		\int_0^{x_\lambda} \frac{\pi \, d x}{\sqrt{\lambda - q(x)}} 
		\left(
		1 + \BigO(x_\lambda^{-1/3} \lambda^{-1/6} )
		\right) .
		\end{aligned}
		\end{equation}
	\end{enumerate}
\end{lemma}
\begin{proof}
	\begin{enumerate}[(i)]
		\item The number $\delta$ is defined by the equation
		\begin{equation}\label{delta.def.2}
		\int_{x_\lambda - \delta}^{x_\lambda} \sqrt{\lambda - q(x)} \, dx =1. 
		\end{equation}
		First we show that $\delta = o(x_\lambda)$ as $\lambda \to \infty$. We proceed by contradiction: let there exist $\varepsilon_0 \in (0,1/2)$ such that $\delta \geq \varepsilon_0 x_\lambda$ for all sufficiently large $\lambda$. Since $q'$ is non-decreasing,
		\begin{equation}\label{delta.def.3}
		\begin{aligned}
		1 &= \int_{x_\lambda - \delta}^{x_\lambda} \sqrt{\lambda - q(x)} \, dx 
		\geq 
		\int_{(1-\varepsilon_0)x_\lambda}^{x_\lambda} \sqrt{\lambda - q(x)} \, dx
		\\
		& \geq
		\sqrt{ q'((1-\varepsilon_0)x_\lambda)} \int_{(1-\varepsilon_0)x_\lambda}^{x_\lambda} \sqrt{x_\lambda - x} \, dx 
		\geq C (\varepsilon_0 x_\lambda)^{3/2},
		\end{aligned}
		\end{equation}
		which leads to the contradiction since $x_\lambda \to \infty$.

		Using that $q'$ is non-decreasing, Taylor polynomial (of the second degree) for $q(x_\lambda - \delta)$, $q''(x)/q'(x) = \BigO(1/x)$ as $x \to \infty$ and $\delta = o(x_\lambda)$, we obtain (with $\xi \in (x_\lambda-\delta,x_\lambda)$)
		\begin{equation}
		\begin{aligned}
		1 & = \int_{x_\lambda - \delta}^{x_\lambda} \sqrt{\lambda - q(x)} \, dx
		\\ &
		= \int_{x_\lambda - \delta}^{x_\lambda} \frac{q'(x)}{q'(x)}\sqrt{\lambda - q(x)} \, dx
		\geq  
		\frac 2 {3 a_\lambda} (\lambda - q(x_\lambda - \delta))^{3/2}
		\\
		& \geq  
		\frac 23 \delta^{3/2} a_\lambda^{1/2} 
		\left(
		1 - \frac{1}{2}\left| \frac{q''(\xi)}{q'(\xi)} \right| 
		\left| \frac{q'(\xi)}{q'(x_\lambda)} \right| \delta 
		\right)^{3/2}
		\\ &
		\geq 
		\frac 23 \delta^{3/2} a_\lambda^{1/2}  
		\left(
		1 - \frac{M \delta}{x_\lambda - \delta}  
		\right)^{3/2}
		\geq
		C \delta^{3/2} a_\lambda^{1/2},  
		\end{aligned}
		\end{equation}
		from which the claim follows. The reasoning for $\delta_1$ is similar but shorter: the analogue of \eqref{delta.def.3} is 
		\begin{equation}
		1 = \int_{x_\lambda}^{x_\lambda+\delta_1} \sqrt{\lambda - q(x)} \, dx \geq
		\sqrt{ q'(x_\lambda)} \int_{x_\lambda}^{x_\lambda+\delta_1} \sqrt{x_\lambda - x} \, dx 
		\geq \frac{2}{3} \delta^{3/2} a_\lambda^{1/2},
		\end{equation}
		from which the claim follows.
		\item We give the proof for $b^+$, the other case being analogous. Since $q$ is increasing, it suffices to show that $q(b^+) > q(x_\lambda + \delta_1)$, which, using the definition of $b^+$, can be rewritten as
		\begin{equation}
		\varepsilon > \frac{q(x_\lambda + \delta_1) - q(x_\lambda)}{q(x_\lambda)}.
		\end{equation}
		Using the mean value theorem, the assumption that $q'$ is non-decreasing and $q'(x)/q(x) = \BigO(1/x)$, we obtain (with $\xi \in (x_\lambda, x_\lambda + \delta_1)$),
		\begin{equation}
		\frac{q(x_\lambda + \delta_1) - q(x_\lambda)}{q(x_\lambda)} 
		= 
		\frac{q'(\xi) \delta_1 }{q(x_\lambda) } 
		=
		\frac{q'(\xi)}{q(\xi)} \frac{q(\xi) \delta_1} {q(x_\lambda)} 
		\leq 
		\frac{M}{x_\lambda} \frac{q(x_\lambda + \delta_1) \delta_1} {q(x_\lambda)}.
		\end{equation}
		Using the mean value theorem and the properties of $q$ again, we get
		\begin{equation}
		\frac{q(x_\lambda + \delta_1)} {q(x_\lambda)} 
		\leq 
		1 + \frac{q'(x_\lambda + \delta_1) \delta_1} {q(x_\lambda)}
		\leq
		1 + \frac M{x_\lambda} \frac{q(x_\lambda + \delta_1) \delta_1} {q(x_\lambda)},
		\end{equation}
		so
		\begin{equation}
		\frac{q(x_\lambda + \delta_1)} {q(x_\lambda)} 
		\leq 
		\frac 1 {1- \frac{M \delta_1}{x_\lambda}}. 
		\end{equation}
		Since $\delta_1 = \BigO(a_\lambda^{-1/3})$, we receive
		\begin{equation}
		\frac{q(x_\lambda + \delta_1) - q(x_\lambda)}{q(x_\lambda)} = \BigO(a_\lambda^{-1/3} x_\lambda^{-1})
		\end{equation}
		and the existence of the desired $\lambda_\varepsilon$ follows.
		\item See the proof of \cite[Lem.\ 5]{Giertz-1964-14}.
		\item See \cite[Eq.\ (14)]{Giertz-1964-14}.
		\item See \cite[Lem.\ 5]{Giertz-1964-14} and its proof.
	\end{enumerate}
\end{proof}

\begin{proposition}\label{prop:u.exp}
	Let $q$ satisfy Assumption \ref{asm:q}, $u$ be as in \eqref{u.def}, $\varepsilon \in (0,1)$, $b^+$ and $\lambda_\varepsilon$ be as in Lemma \ref{lem:u.basic}.\ref{b.lam} and $C_1>0$ from Lemma \ref{lem:u.basic}.\ref{u.exp.1}.
	Then 
	\begin{equation}\label{u.exp.2}
	|u(x)| \leq \frac{C_1(1+\frac 1 \varepsilon)^{1/4}}{q(x)^{1/4}} 
	\exp \left(
	-\frac 23 \left( \frac{\varepsilon}{1 + \varepsilon}\right)^{3/2} \frac{q(x)^{3/2}}{q'(x)}
	\right)
	\end{equation}
	holds for all $\lambda > \lambda_{\varepsilon}$ and $x > b^+$.
\end{proposition}

\begin{proof}
	For every $\varepsilon \in (0,1)$, the bound \eqref{u.int.bound} is valid for all $x > b^+$ and $\lambda > \lambda_\varepsilon$, see Lemma \ref{lem:u.basic}. Since $q'$ is non-decreasing,
	\begin{equation}
	\int_{x_\lambda}^x \sqrt{q(s) - \lambda} \, ds 
	\geq 
	\frac 1{q'(x)} \int_{x_\lambda}^x q'(s)\sqrt{q(s) - \lambda} \, ds
	= 
	\frac 23 \frac{(q(x) - \lambda)^{3/2}}{q'(x)}.
	\end{equation}
	The claim then follows from $\lambda < \frac 1 {1+\varepsilon} q(x)$, which is implied by $b^+ < x $.
\end{proof}

We turn to the analysis of $u$ before the turning point.

\begin{proposition}\label{prop:u.inside}
	Let $q$ satisfy Assumption \ref{asm:q}, $u$ be as in \eqref{u.def}, $\varepsilon \in (0,1)$, $b^-$ and $\lambda_\varepsilon$ as in Lemma \ref{lem:u.basic}.\ref{b.lam} and let $v \in W^{1,1}_{\rm loc}(\R)$.
	Then there exists $C_2 > 0$, independent of $\varepsilon$, such that
	\begin{equation}\label{eHermiteInsideWeak}
	\begin{aligned}
	& \left|\int_0^{b^-} v(x) u(x)^2\,dx - \int_0^{b^-} \frac{\pi v(x)\, dx}{\sqrt{\lambda - q(x)}}\right|
	\\
	& \qquad \qquad \leq 
	\frac {1}{\varepsilon \lambda}
	\left(\left(C_2 + \frac{3}{4}|\log \eps|\right)
	\|v\|_{L^{\infty}((0,b^-))}
	+ 
	\|v'\|_{L^1((0,b^-))}
	\right)
	\end{aligned}
	\end{equation}
	
	holds for all $\lambda > \lambda_\varepsilon$.
\end{proposition}

\begin{proof}
	Regarding Lemma \ref{lem:u.basic}.\ref{u2.in}, we need to estimate
	\begin{equation}\label{sin.R1.est}
	\int_0^{b^-} v(x) \frac{\sin 2 \zeta(x) + R_1(\zeta(x))}{-\zeta'(x)} \, dx.
	\end{equation}
	To simplify notations, norms $\|\cdot\|_{L^{\infty}((0,b^-))}$ and $\|\cdot\|_{L^{1}((0,b^-))}$ are denoted by $\|\cdot\|_\infty$, $\|\cdot\|_1$, respectively.
	First, integrating by parts and using that $\zeta'$ is decreasing and $\zeta'(b^-)^2 = \lambda - q(b^-) = \varepsilon \lambda$, we obtain
	\begin{equation}\label{sin.est}
	\begin{aligned}
	\left|
	\int_0^{b^-} v(x) \frac{\sin 2 \zeta(x)}{-\zeta'(x)} \, dx
	\right|
	&\leq 
	\frac 12
	\left(
	\left|
	\left[\frac{v(x)}{\zeta'(x)^2} \cos 2 \zeta(x) \right]_0^{b^-}
	\right|
	\right.
	\\
	& \quad
	+
	\left. 
	\left|
	\int_0^{b^-} \cos 2 \zeta(x)  \frac{v'(x) \zeta'(x) - 2 v(x) \zeta''(x)}{\zeta'(x)^3} \, dx
	\right|
	\right)
	\\
	& \leq
	\frac12 
	\left(
	\frac{2 \|v\|_\infty}{\varepsilon \lambda} + \frac{\|v'\|_1}{\varepsilon \lambda} 
	+ \|v\|_\infty (\zeta'(b^-)^{-2}-\zeta'(0)^{-2})
	\right)
	\\
	& \leq
	\frac{3\|v\|_\infty}{2\varepsilon\lambda} + \frac{\|v'\|_1}{\varepsilon \lambda}.
	\end{aligned}
	\end{equation}
	Concerning the second term in \eqref{sin.R1.est}, recall that $|R_1(\zeta)| \leq 1/(2\zeta)$, thus
	\begin{equation}
	\begin{aligned}
	\int_0^{b^-} |v(x)|  \frac{|R_1(\zeta(x))|}{|\zeta'(x)|} \, dx
	\leq \frac{\|v\|_\infty}2 \int_0^{b^-} \frac{-\zeta'(x) \, dx}{\zeta'(x)^2 \zeta(x)} 
	\leq \frac{\|v\|_\infty}{2 \varepsilon \lambda} \log \frac{\zeta(0)}{\zeta(b^-)}.
	\end{aligned}
	\end{equation}
	Next,
	\begin{equation}
	\zeta(0) = \int_0^{x_\lambda}  \sqrt{\lambda - q(x)} \, dx 
	\leq \sqrt{\lambda} x_{\lambda}
	\end{equation}
	and, using that $q'$ is non-decreasing,  $q'(x)/q(x) = \BigO(1/x)$ and $q(x_\lambda)=\lambda$,
	\begin{equation}
	\begin{aligned}
	\zeta(b^-) & = \int_{b^-}^{x_\lambda}  \sqrt{\lambda - q(x)} \, dx 
	\geq 
	\frac{1}{q'(x_\lambda)} \int_{b^-}^{x_\lambda} q'(x)  \sqrt{\lambda - q(x)} \, dx 
	=
	\frac 23 \frac{(\lambda \varepsilon)^{3/2} }{ q'(x_\lambda)}
	\\
	&
	\geq M \sqrt \lambda x_\lambda \varepsilon^{3/2}
	. 
	\end{aligned}
	\end{equation}
	Hence, altogether we have
	\begin{equation}
	\int_0^{b^-} |v(x)|  \frac{|R_1(\zeta(x))|}{|\zeta'(x)|} \, dx
	\leq
	\frac{\|v\|_\infty}{2 \varepsilon \lambda} \log
	\frac{\zeta(0)}{\zeta(b^-)} 
	\leq 
	\frac{\|v\|_\infty}{2 \varepsilon \lambda} \log \frac{M}{\varepsilon^{3/2}}.
	\end{equation}
	Combining this with \eqref{sin.est} in \eqref{sin.R1.est} gives the result.
\end{proof}

\begin{corollary}\label{cor:u.bpm}
	Let $q$ satisfy Assumption \ref{asm:q}, $u$ be as in \eqref{u.def}, $\varepsilon \in (0,1)$ and $b^\pm$ as in Lemma \ref{lem:u.basic}.\ref{b.lam}.   Then, as $\lambda \to +\infty$,
	\begin{equation}
	\int_{b^-}^{b^+}u(x)^2dx = \int_{b^-}^{x_\lambda} \frac{\pi \, dx}{\sqrt{\lambda - q(x)}} 
	\left(
	1 + \BigO(x_\lambda^{-1/3} \lambda^{-1/6}) 
	\right).
	\end{equation}
\end{corollary}

\begin{proof}
	From \eqref{u.norm} and \eqref{u.exp.2},
	\begin{equation}
	\begin{aligned}
	\int_{b^-}^{b^+}u(x)^2\,dx 
	&= 
	\int_{0}^{\infty} u(x)^2\,dx
	- \int_{0}^{b^-} u(x)^2\,dx
	- \int_{b^+}^{\infty} u(x)^2\,dx
	\\
	& =
	\int_{0}^{x_\lambda} \frac{\pi \, dx}{\sqrt{\lambda - q(x)}} 
	+ \BigO(x_\lambda^{2/3} \lambda^{-2/3}) 
	- \int_{0}^{b^-} \frac{\pi \, dx}{\sqrt{\lambda - q(x)}} 
	\\
	& \quad
	+ \int_{0}^{b^-} \frac{\pi \, dx}{\sqrt{\lambda - q(x)}} 
	- \int_{0}^{b^-} u(x)^2\,dx
	+ \BigO(e^{-C \lambda^{1/2}}).
	\end{aligned}
	\end{equation}
	Next, by Proposition \ref{prop:u.inside} with $v=1$,
	\begin{equation}
	\begin{aligned}
	\int_{b^-}^{b^+}u(x)^2\,dx 
	&= 
	\int_{b^-}^{x_\lambda} \frac{\pi \, dx}{\sqrt{\lambda - q(x)}} 
	+ \BigO(x_\lambda^{2/3} \lambda^{-2/3}) 
	+ \BigO(\lambda^{-1}).
	\end{aligned}
	\end{equation}
	Finally, since $q'$ is non-decreasing and $q'(x)/q(x) = \mathcal{O}(1/x)$,
	\begin{equation}
	\int_{b^-}^{x_\lambda} \frac{dx}{\sqrt{\lambda - q(x)}} = \int_{b^-}^{x_\lambda} \frac{q'(x)\,dx}{q'(x)\sqrt{\lambda - q(x)}} \geq 2\sqrt{\varepsilon} \frac{\sqrt{\lambda}}{q'(x_\lambda)} \geq M \frac{x_\lambda}{\sqrt{\lambda}},
	\end{equation}
	hence the claim follows.
\end{proof}

\bibliographystyle{acm}
\bibliography{VaryingRates}

\end{document}